\documentclass[12pt,reqno]{amsart}     
\usepackage[english]{babel}
\usepackage{amssymb}
\usepackage[top=2.5cm,right=2.4cm,left=2.4cm,bottom=2.2cm]{geometry} 
\usepackage[colorlinks=true,linkcolor=black,citecolor=black]{hyperref}

\title[Quantum weighted projective spaces]{Quantum weighted projective and lens spaces}

\author[F.~D'Andrea and G.~Landi]{Francesco D'Andrea and Giovanni Landi}

\address[F.~D'Andrea]{Dipartimento di Matematica e Applicazioni, Universit\`a di Napoli ``Federico II''
         and I.N.F.N. Sezione di Napoli, Complesso MSA, Via Cintia, 80126 Napoli, Italy}
\email{francesco.dandrea@unina.it}
\address[G.~Landi]{Matematica, 
Universit\`{a} di Trieste, Via A.~Valerio~12/1, I-34127 Trieste, Italy and I.N.F.N. Sezione di Trieste, Trieste, Italy.}
\email{landi@units.it}

\thanks{\hspace*{-\parindent}\textit{Date:} February 2015. 
\\[2pt]
2010 \textit{Mathematics Subject Classification}. Primary: 20G42; Secondary: 58B32; 58B34.
\\[2pt]
\textit{Key words and phrases.}
Noncommutative geometry, Fredholm modules, weighted quantum projective spaces, quantum lens spaces, quantum principal circle bundles.
\\[5pt]
\textit{Thanks.}
This work was partially supported by the Italian Project ``Prin 2010-11 -- Operator Algebras, Noncommutative Geometry and Applications''.
F.D.~was partially supported by UniNA and Compagnia di San Paolo under the Program STAR 2013.}

\allowdisplaybreaks[4]
\numberwithin{equation}{section}
 
\newtheorem{prop}{Proposition}[section]
\newtheorem{thm}[prop]{Theorem}
\newtheorem{lemma}[prop]{Lemma}
\newtheorem{cor}[prop]{Corollary}
\newtheorem{df}[prop]{Definition}
\newtheorem{ex}[prop]{Example}

\theoremstyle{definition}
\newtheorem{rem}[prop]{Remark}

\newcommand{\U}{\mathrm{U}}
\newcommand{\A}{\mathcal{O}}
\newcommand{\B}{\mathcal{B}}
\newcommand{\HH}{\mathcal{H}}

\newcommand{\N}{\mathbb{N}}
\newcommand{\Z}{\mathbb{Z}}
\newcommand{\R}{\mathbb{R}}
\newcommand{\C}{\mathbb{C}}
\newcommand{\CP}{\mathbb{C}\mathrm{P}}
\newcommand{\WP}{\mathbb{P}}
\newcommand{\inner}[1]{\left<#1\right>}
\newcommand{\ket}[1]{\left|\smash[t]{#1}\right>}
\newcommand{\tr}{\mathrm{Tr}}

\newcommand{\sqbn}[2]{\genfrac{[}{]}{0pt}{}{#1}{#2}_{\!q}}
\newcommand{\mc}{\mathcal}
\newcommand{\mr}{\mathrm}

\newcommand{\mo}{\mathcal} 
\newcommand{\tbrace}[2]{\genfrac{ \{ }{ \} }{0pt}{1}{#1}{#2}_{\!q}}
\newcommand{\dbrace}[2]{\genfrac{ \{ }{ \} }{0pt}{}{#1}{#2}_{\!q}}
\newcommand{\mv}{\underline}
\newcommand{\vv}{\boldsymbol}
\renewcommand{\Re}{\mathrm{Re}}
\renewcommand{\Im}{\mathrm{Im}}
\newcommand{\om}{\hspace{-1pt}:\hspace{-1pt}}

\begin{document}

\begin{abstract}
We generalize to quantum weighted projective spaces in any dimension previous results of us on K-theory and K-homology of
quantum projective spaces `tout court'. For a class of such spaces, we explicitly construct families of  
Fredholm modules, both bounded and unbounded (that is spectral triples),
and prove that they are linearly independent in the K-homology of the corresponding $C^*$-algebra.
We also show that the quantum weighted projective spaces are base spaces of quantum principal 
circle bundles whose total spaces are quantum lens spaces.
We construct finitely generated projective modules associated with the principal bundles and pair them with the Fredholm modules, 
thus proving their non-triviality.

\bigskip\bigskip\bigskip

\centerline{\emph{Dedicated to Marc Rieffel on the occasion of his 75th birthday}}

\end{abstract}

\maketitle

\setcounter{tocdepth}{1}
\tableofcontents
\setlength{\parskip}{4pt}


\section{Introduction}\label{se:intro}

This paper deals with the geometry of quantum \emph{weighted} projective spaces. 
For any weight vector $\mv{\ell}=(\ell_0,\ldots,\ell_n)$, the quantum weighted projective space $\WP_q(\mv\ell)$ is the quotient of the 
odd-dimensional quantum sphere $S^{2n+1}_q$ by a weighted action of U(1). The latter acts by automorphisms on the coordinate algebra $\A(S^{2n+1}_q)$: on generators $z_i$ just by  
$$
\alpha_t(z_i)=t^{\ell_i}z_i \;,\qquad\forall\;t\in\U(1),\,i=0,\ldots,n 
$$
(cf.~\S\ref{sec:due} for details).
The fixed point algebra $\A(\WP_q(\mv\ell))$ is defined to be the coordinate algebra of $\WP_q(\mv\ell)$. For $\ell_0=\ldots=\ell_n=1$, one gets in this way the quantum projective space $\CP^n_q$.

Quantum weighted projective lines $\WP_q(\ell_0,\ell_1)$ were studied in \cite{BF12}, with a particular attention to quantum teardrops (for which $\ell_0=1$), and with more generality in \cite{AKL14}. The present paper is devoted to the case of arbitrary dimension. While the most general case of an arbitrary weight vector seems intractable at the moment, we select a particularly nice class of weight vectors which allows us to push things considerably. We focus on weight vectors $\mv{\ell}=(\ell_0,\ldots,\ell_n)$ for which the classical $\WP(\mv\ell)$ is homeomorphic to the ordinary projective space $\CP^n$. (An algebraic characterization of these weight vectors is in \S\ref{sec:tre}.)

By a fortunate stroke of serendipity, for such a class of weight vectors the coordinate algebra $\A(\WP_q(\mv\ell))$ is generated by simple elements, of length 2 
(see \S\ref{sec:quattro}). This gives an alternative characterization of such a class of weight vectors: one belongs to this class if and only if the corresponding coordinate algebra has a set of generators with length no greater than 2.

We next generalize some of the results of \cite{DL10} to $\WP_q(\mv{\ell})$.
After an interlude on irreducible representations in \S\ref{sec:irreps}, we construct a family of $1$-summable Fredholm 
modules in \S\ref{sec:Khom}, and a class of Dirac operators (and then spectral triples) in \S\ref{sec:otto}. 

In~\S\ref{sec:sei}, a quantum weighted projective space 
$\WP_q(\mv{\ell})$ is shown to be the base space of a quantum principal circle bundle whose total space is a quantum 
lens space $L_q(p;\mv{\ell})$, with explicit examples in \S\ref{sec:sette}. 
For the weighted projective lines this was done in \cite{BF12} and in generality in \cite{AKL14}. 

In \S\ref{sec:nove}, we construct projections in the $C^*$-algebra of quantum weighted projective spaces and prove that the Fredholm modules of \S\ref{sec:Khom} are linearly independent in K-homology. 

We finish by mentioning possible uses of quantum projective lines (and more general quantum projective spaces) for Chern-Simons theories. These theories were generalised in \cite{BT13} to circle-bundles over orbifolds (such as weighted projective spaces). However, these applications go beyond the scope of the present paper and should await a future time. 

\smallskip

\noindent\textit{General notations.}
By a \emph{$*$-algebra} we mean a complex unital involutive associative algebra
and by a representation of a $*$-algebra we always mean a unital $*$-representation.
We will denote by $\mv{m}=(m_0,m_1,\ldots,m_n)$ a vector with $n+1$ components (labelled from $0$ to $n$), and by
$\vec{m}=(m_1,\ldots,m_n)$ a vector with $n$ components (labelled from $1$ to $n$), when $n\geq 1$ is the complex
dimension of the space considered.

\smallskip

\noindent\textit{Acknowledgments.}
We are grateful to A.~De Paris, M.~Frank and T.~Brzezi{\'n}ski for useful discussions and correspondences.
Part of the work was done at the Hausdorff Research Institute for Mathematics in Bonn during the 2014 Trimester Program ``Non-commutative Geometry and its Applications". We thank the organizers of the Program for the invitation and all people at HIM for the nice hospitality.

\section{Weighted projective spaces}\label{sec:tre}

Classical weighted projective spaces $\WP(\mv\ell)$ are among the best known examples of projective toric varieties. As quotient spaces they have a natural orbifold structure and include, in complex dimension 1, the orbifolds named teardrops by Thurston in \cite{Thu80}. We start by recalling some basic facts about their classification as done in \cite{BFNR13}.

A \emph{weight vector} $\mv{\ell}=(\ell_0,\ldots,\ell_n)$ is a finite sequence of positive integers, called \emph{weights}.
A weight vector is \emph{normalized} if for any prime $p$ at least two weights are not divisible by $p$.
One says that a weight vector is \emph{coprime} if $\gcd(\ell_0,\ldots,\ell_n)=1$; and it is \emph{pairwise coprime} if $\gcd(\ell_i,\ell_j)=1$, for all $i\neq j$.
For $n=1$, the only normalized weight vector is $(1,1)$; if $n=2$, a weight $\mv{\ell}$ is pairwise coprime if and only if it is normalized; if $n\geq 3$, every pairwise coprime weight vector is normalized but the
converse is not true, e.g.~$(1,1,\ell_2,\ldots,\ell_n)$ is normalized but not necessarily pairwise coprime.

Fixed a weight vector $\mv{\ell}=(\ell_0,\ldots,\ell_n)$, an action of $t\in \C^*$ on $z=(z_0,\ldots,z_n)\in \C^{n+1}\smallsetminus\{0\}$
is given by $z\mapsto (t^{\ell_0}z_0,\ldots,t^{\ell_n}z_n)$. If the action is restricted to the subgroup $\U(1)\subset  \C^*$, 
the unit sphere $S^{2n+1}\subset \C^{n+1}$ is an invariant submanifold.

The quotient $\C^{n+1}\smallsetminus\{0\} / \C^*$ yields the weighted projective space $\WP(\mv{\ell})=\WP(\ell_0,\ldots,\ell_n)$; 
when $\ell_0=\ldots=\ell_n=1$, this is the ordinary complex projective space $\CP^n$.
In general there is an embedding $\WP(\mv{\ell})\hookrightarrow\CP^n$ realizing $\WP(\mv{\ell})$ as a complex projective (toric) variety.
As a topological space, $\WP(\mv{\ell})$ is homeomorphic to the quotient $S^{2n+1}/\U(1)$ with respect to the weighted action defined above, a quotient which has a natural structure of orbifold.

It is known that two weighted projective spaces are isomorphic as
projective varieties if and only if they are homeomorphic \cite{BFNR13}. Moreover,
every isomorphism class can be represented by a normalized weight,
and two such spaces are isomorphic if and only if they have the same normalized weights, up to order.
As a corollary, for $n=1$ every such space is isomorphic to $\CP^1$. Let us stress that, despite this,
$\WP(\ell_0,\ell_1)$ has orbifold singularities in every case except $\ell_0=\ell_1=1$; the spaces $\WP(\ell_0,\ell_1)$ are all
homeomorphic to $\CP^1$, but not isomorphic as orbifolds.

It is useful to have in mind few basic examples.

\begin{ex}
The map $S^3\to\R^3$, $(z_1,z_2)\mapsto (x_1,x_2,x_3):=\big(\Re(z_1^*z_2^3),\Im(z_1^*z_2^3),z_1z_1^*\big)$,
factors to a homeomorphism between $\WP(3,1)$ and the variety in $\R^3$ defined by the equation 
$x_1^2+x_2^2=x_3(1-x_3)^3$.
The latter has an orbifold singularity (a cusp) at $x=(0,0,1)$. As a toric variety,
the map \mbox{$[z_1\om z_2] \mapsto [z_1\om z_2^3]$} is an isomorphism $\WP(3,1)\to\CP^1$.
\end{ex}

\noindent
As the $\WP(3,1)$ example shows, $\WP(\mv{\ell})$ may be isomorphic to $\CP^n$ as complex algebraic
variety and, at the same time, have orbifold singularities.

\begin{ex}
There is an isomorphism of algebraic varieties between $\WP(1,1,2)$ and the cone
in $\CP^3$ with equation $y_0y_2-y_1^2=0$; the isomorphism is given by
$$
[z_0\om z_1\om z_2]\mapsto [y_0\om y_1\om y_2\om y_3]\om =[z_0^2\om z_0z_1\om z_1^2\om z_2] \;,
$$
\end{ex}

As a preliminary step, we investigate the conditions on the weight vector $\mv{\ell}$ for which \mbox{$\WP(\mv{\ell})\simeq \CP^n$}.
Firstly, 
the observation that $\WP(\mv{\ell})\simeq\WP(m\mv{\ell})$ for all $m\geq 1$ allows one to consider coprime weight vectors
only (cf.~Lemma~\ref{lemma:obs} for the quantum case).

\begin{df}
If $\mv{\ell}=(\ell_0,\ldots,\ell_n)$ is a weight vector, we denote by $\mv{\ell}^{\,\sharp}$ the weight vector
whose $i$-th component is equal to $\prod_{j\neq i}\ell_j$, for all $i=0,\ldots,n$.

For all $p\geq 1$ and $0\leq k\leq n$, we further define an operation $\mo{M}_k(p)$ on weight vectors 
as follows: we let $\mo{M}_k(p)\mv{\ell}$ be the weight vector whose $i$-th component is equal to $\ell_i$ if $i=k$, and
to $p\ell_i$ otherwise.

The left inverse of $\mo{M}_k(p)$ is the operation $\mo{D}_k(p)$ partially defined as follows: if $p$ divides 
$\ell_j$ for all $j\neq k$, then $\mo{D}_k(p)\mv{\ell}$ is defined as the weight vector whose $i$-th component 
is equal to $\ell_i$ if $i=k$, and to $p^{-1}\ell_i$ otherwise.
\end{df}
The map $\mv{\ell}\mapsto\mv{\ell}^{\,\sharp}$ is not exactly an involution, but it satisfies
$$
(\mv{\ell}^{\,\sharp})^{\,\sharp}=m \hspace{1pt} \mv{\ell}
$$
with $m:=(\ell_0\ell_1\cdots\ell_n)^{n-1}$.
Therefore, from the above mentioned fact that $\WP(\mv{\ell})\simeq\WP(m\mv{\ell})$, 
every isomorphism class of weighted projective spaces can be obtained from
a weight that is in the image of this map. 
A similar result holds for the quantum case as well, cf.~Lemma~\ref{lemma:obs} below.

\begin{df}
Let $\mv{\ell}$ be a weight vector.
The \emph{multiplication} $\mo{M}_k(p)$ is \emph{admissible} on $\mv{\ell}$ if
$p$ is a prime number and is not a divisor of $\ell_k$.
The \emph{division} $\mo{D}_k(p)$ is \emph{admissible} on $\mv{\ell}$ if
$p$ is prime, divides $\ell_i$ for all $i\neq k$ and does not divide $\ell_k$.
\end{df}

Classically, two (coprime) weight vectors correspond to isomorphic weighted projective spaces if and
only if one can be obtained from the other with an iterated application (in any order) of admissible
multiplications and divisions \cite{BFNR13}.
So for example, one has $\WP(1,2,2)\simeq\WP(2,3,6)\simeq\CP^2$ while $\WP(1,1,2)\not\simeq\CP^2$
(in particular, $\mo{D}_0(2)$ is admissible in the former case, while $\mo{M}_2(2)$ is not admissible in the latter case).

\begin{lemma}
A weight vector
$\mv{\ell}=\mv{m}^{\sharp}$ is coprime if and only if $\mv{m}$ is pairwise coprime.
\end{lemma}
\begin{proof}[Proof of ``$\,\Leftarrow$'']
Note that $\gcd(\ell_i,\ell_j)=\gcd(m_i,m_j)\prod_{k\neq i,j}m_k$ for all $i\neq j$.

Let $\mv{m}$ be pairwise coprime. We prove by induction that
$\gcd(\ell_0,\ell_1,\ldots,\ell_k)=\prod_{i>k}m_i$ for all $1\leq k\leq n$ (with the convention
that empty products are $1$).
It follows from the equation above when $k=1$. Assume it is true for some $k\in\{1,\ldots,n-1\}$. Then
$$
\gcd(\ell_0,\ell_1,\ldots,\ell_{k+1})=
\big(\gcd(\ell_0,\ell_1,\ldots,\ell_k),\ell_{k+1}\big)=
\gcd(1,m_0\cdots m_{k})\prod_{i>k+1}m_i=
\prod_{i>k+1}m_i \;.
$$

\noindent
\textit{Proof of ``$\,\Rightarrow$''.}
Fix $i,j$ with $i\neq j$.
Since $m_i$ divides $\ell_k$ for all $k\neq i$ and $m_j$ divides $\ell_i$,
$\gcd(m_i,m_j)$ divides $\ell_k$ for all $k$. Hence if $\mv{\ell}$ is coprime,
$\gcd(m_i,m_j)=1$, for all $i\neq j$.
\end{proof}

\begin{thm}\label{thm:6}
Let $\mv{\ell}$ be coprime. Then $\WP(\mv{\ell})\simeq\CP^n$ if and only if there is a pairwise coprime
weight vector $\mv{p}$ such that $\mv{\ell}=\mv{p}^{\,\sharp}$.
\end{thm}

\begin{proof}[Proof of ``$\,\Rightarrow$'']
Let $\mv{p}$ be pairwise coprime and $\mv{\ell}=\mv{p}^{\,\sharp}$.

For any prime $m$, $\mo{M}_k(m)$ is admissible if and only if $m$ does
not divide $\ell_k=\prod_{i\neq k}p_i$. It follows that $\gcd(m,p_i)=1$ for all $ i\neq k$.
The effect of $\mo{M}_k(m)$ is multiplying $p_k$ by $m$. Note that the new weight vector has still the form
$\mv{\ell}'=\mv{p}'^{\,\sharp}$ with $\mv{p}'$ pairwise coprime.

The division $\mo{D}_k(m)$ is admissible if and only if $m$ divides $\ell_j=\prod_{i\neq j}p_i$ for all $j\neq k$
and does not divide $\ell_k$. It follows that $\gcd(m,p_i)=1$ for all $i\neq k$, and $m$ divides $p_k$.
The effect of $\mo{D}_k(m)$ is to divide $p_k$ by $m$. The new weight vector has still the form
$\mv{\ell}'=\mv{p}'^{\,\sharp}$ with $\mv{p}'$ pairwise coprime.
Thus, any iterated application of admissible multiplications or divisions transforms $\mv{\ell}=\mv{p}^{\,\sharp}$ 
into $\mv{\ell}'=\mv{p}'^{\,\sharp}$, with $\mv{p}'$ pairwise coprime. Starting with $\mv{p}=\mv{\ell}=(1,\ldots,1)$ 
this proves the implication ``$\,\Rightarrow$''.

\noindent
\textit{Proof of ``$\,\Leftarrow$''.}
Let $\mv{\ell}=\mv{p}^{\,\sharp}$ with $\mv{p}$ pairwise coprime.
From the discussion above, $\mo{D}_k(m)$ is admissible for any prime factor $m$ of $p_k$.
By repeated application of admissible divisions, we can transform $\mv{p}$ into $(1,\ldots,1)$,
and this proves that $\WP(\mv{\ell})\simeq\CP^n$.
\end{proof}

Quantum spaces $\WP_q(\mv{\ell})$ with weight vectors as in Theorem~\ref{thm:6} form a nice class of spaces and will be of primary interest
in the rest of the paper. The class include any quantum weighted projective line (and in particular all quantum teardrops),
since for $n=1$ any weight vector is such that $(\ell_0,\ell_1)=(\ell_1,\ell_0)^{\,\sharp}$.

\section{Quantum weighted projective and lens spaces}\label{sec:due}

Fix an integer $n\geq 1$.
The coordinate algebra $\A(S^{2n+1}_q)$ of the $2n+1$-dimensional quantum sphere
is generated by $2(n+1)$ elements $\{z_i,z_i^*\}_{i=0,\ldots,n}$ with relations \cite{VS91} (see also \cite{We00}):
\begin{subequations}\label{eq:defrel}
\begin{align}
z_iz_j &=q^{-1}z_jz_i &&\forall\;0\leq i<j\leq n \;, \\
z_i^*z_j &=qz_jz_i^* &&\forall\;i\neq j \;, \label{eq:defrelB} \\
[z_i^*,z_i] &=(1-q^2)\sum\nolimits_{j=i+1}^n z_jz_j^* &&\forall\;i=0,\ldots,n-1 \;,  \label{eq:defrelC} \\
[z_n^*,z_n] &=0 \;, \qquad \; \\
z_0z_0^*+z_1z_1^* &+\ldots+z_nz_n^*=1 \;. \label{eq:defrelE}
\end{align}
\end{subequations}
We use the notations of \cite{DL10}. 
We have denoted by $q$ the deformation parameter, and assume that $0<q<1$.
The original notation of \cite{VS91} is obtained by setting $q=e^{h/2}$;
the generators $x_i$ used in~\cite{HL04} are related to ours by
$x_i=z_{n+1-i}^*$ and replacing $q\to q^{-1}$.

\smallskip

Let $\mv{\ell}$ be a weight vector and $\U(1) = \{ t\in\C , \, |t|=1\}$. 
An action by $*$-automorphisms $\alpha:\U(1)\to\mr{Aut}\,\A(S^{2n+1}_q)$ is defined on generators by:
\begin{equation}\label{eq:act}
\alpha_t(z_i):=t^{\ell_i}z_i \;,\qquad\forall\;i=0,\ldots,n \, .
\end{equation}
Invariant elements form a $*$-subalgebra
\begin{equation}\label{eq:WP}
\A(\WP_q(\mv{\ell})):=\big\{a\in\A(S^{2n+1}_q):\alpha_t(a)=a \;\forall\; t\in \U(1)\big\} \;.
\end{equation}
The virtual underlying quantum space is called \emph{quantum weighted (complex) projective space}
$\WP_q(\mv{\ell})$ in \cite{BF12} (with a slightly different notation for quantum spheres that the one used there). 
In particular, $\WP_q(\,\underbrace{\!1,\ldots,1\!}_{n}\,)=\CP^n_q$ is a quantum projective
space `tout court'.

\smallskip

Next, let $p$ be a positive integer and $\Z_p\subset \U(1)$ the subgroup of $p$-th roots of unity.
We call \emph{quantum lens space} $L_q(p;\mv{\ell})$ the virtual space underlying the algebra
\begin{equation}\label{eq:Lq} 
\A(L_q(p;\mv{\ell})):=\big\{a\in\A(S^{2n+1}_q):\alpha_t(a)=a \;\forall\; t\in \Z_p\big\} \;.
\end{equation}

The above definition of quantum lens space in arbitrary dimension was introduced in \cite{HL03},  
in the framework of graph $C^*$-algebras. 

Clearly $L_q(1,\mv{\ell})=S^{2n+1}_q$ and \eqref{eq:WP} is a $*$-subalgebra of \eqref{eq:Lq}, for any $p$:
$$
\A(\WP_q(\mv{\ell})) \hookrightarrow \A(L_q(p;\mv{\ell})) \;.
$$
We shall have a closer look at this inclusion later on in \S\ref{sec:sei}.  

\medskip

As for $q=1$, we do not lose generality by working with coprime weights. Indeed,
\begin{lemma}\label{lemma:obs}
For all $m\geq 1$, it holds that $\A(\WP_q(\mv{\ell}))\simeq \A(\WP_q(m\mv{\ell}))$.
\end{lemma}

\begin{proof}
Let $\alpha_{\mv{\ell}}(t)$ be the action in \eqref{eq:act}.
The inclusion $\A(\WP_q(\mv{\ell}))\subset\A(\WP_q(m\mv{\ell}))$ is obvious:
invariance under $\alpha_{\mv{\ell}}(t)$ for $t\in \U(1)$ implies invariance under
$\alpha_{\mv{\ell}}(t^m)=\alpha_{m\mv{\ell}}(t)$ for $t\in \U(1)$.

The opposite inclusion follows from surjectivity of the map $\U(1)\to \U(1)$, $t\mapsto t'=t^m$:
if $\alpha_{m\mv{\ell}}(t)(a)=a$ for $t\in \U(1)$, then
$\alpha_{\mv{\ell}}(t')(a)=a$ for $t'=t^m\in \U(1)$.
\end{proof}

\subsection{Generators of the algebra}\label{sec:quattro}

Let us start again from the generators $\{z_i,z_i^*\}_{i=0,\ldots,n}$ of the sphere algebra $\A(S^{2n+1}_q)$. 
To simplify the notations, we also denote by $z_i^{-1}$ the adjoint $z_i^*$ (there is no ambiguity,
since $z_i$ is not invertible).
For $\mv{k}=(k_0,\ldots,k_n)\in\Z^{n+1}$, let
$$
z^{\mv{k}}:=z_0^{k_0}z_1^{k_1}\ldots z_n^{k_n} \;.
$$
Let $\mv{x}\cdot\mv{y}=x_0y_0+\ldots+x_ny_n$ be the Euclidean inner product.
Next lemma is true for an arbitrary weight vector $\mv{\ell}$.

\begin{lemma}\label{lemma:2}
For any weight vector $\mv{\ell}$, 
a set of generators for the weighted projective space algebra $\A(\WP_q(\mv{\ell}))$ is given by the elements:
\begin{align*}
z_i^*z_i \;,\quad &\forall\;i=0,\ldots,n \\[2pt]
z^{\mv{k}} \;,\quad &\forall\;\mv{k}\in\Z^{n+1}\; \quad \textup{s.t.}\quad \;\mv{\ell}\cdot\mv{k}=0 \,  .
\end{align*}
\end{lemma}

\begin{proof}
A linear basis of $\A(S^{2n+1}_q)$ is given by monomials
$$
z_0^{j_0}z_1^{j_1}(z_1^*)^{k_1}\ldots z_n^{j_n}(z_n^*)^{k_n} \;,\qquad
(z_0^*)^{k_0}z_1^{j_1}(z_1^*)^{k_1}\ldots z_n^{j_n}(z_n^*)^{k_n} \;,
$$
where $j_0,\ldots,j_n,k_0,\ldots,k_n$ are non-negative integers.
Using the commutation rules of $S^{2n+1}_q$, these can be always rewritten as sums of products of elements $z_i^*z_i$, which clearly belong to $\A(\WP_q(\mv{\ell}))$,
and elements $z^{\mv{k}}$.
Under the action \eqref{eq:act}:
$$
z^{\mv{k}}\mapsto t^{\mv{k}\hspace{1pt}\cdot\hspace{1pt}\mv{\ell}}z^{\mv{k}} \;,
$$
hence $z^{\mv{k}}$ belongs to $\A(\WP_q(\mv{\ell}))$  if and only if $\mv{k}\cdot\mv{\ell}=0$.
\end{proof}

\begin{df}
The number of non-zero components of $\mv{k}$ is called the \emph{length} of $z^{\mv{k}}$. 
\end{df}

With $\ell_{i:j}:=\ell_i/\mr{gcd}(\ell_i,\ell_j)$, consider the invariant elements\vspace{2pt}
\begin{equation}\label{eq:conj}
\xi_{i,j}:=(z_i^*)^{\ell_{j:i}}z_{\smash[t]{j}}^{\ell_{i:j}}\;,\qquad\forall\;i,j=0,\ldots,n \;.\vspace{5pt}
\end{equation}

\begin{lemma}
Any length $2$ element $z^{\mv{k}}\in\A(\WP_q(\mv{\ell}))$ is a power of some $\xi_{i,j}$ as in \eqref{eq:conj}.
\end{lemma}

\begin{proof}
Let $i\neq j$. A length $2$ element $z^{\mv{k}}=z_i^{\smash[b]{k_i}}z_{j}^{\smash[b]{k_j}}$
is invariant under the coaction \eqref{eq:act} if and only if $k_i\ell_{i:j}+k_j\ell_{j:i}=0$.
Since $\ell_{i:j}$ and $\ell_{j:i}$ are coprime, this implies that $k_i=m\ell_{j:i}$ and $k_j=-m\ell_{i:j}$
for some $m\in\Z$. From \eqref{eq:defrel},
it follows that $z^{\mv{k}}=z_i^{\smash[b]{m\ell_{j:i}}}z_{j}^{\smash[b]{-m\ell_{i:j}}}$ is, modulo a multiplicative coefficient,
either $(\xi_{i,j})^{|m|}$ or $(\xi_{j,i})^{|m|}$ (depending on the sign of $m$).
\end{proof}

Clearly, $\xi_{i,i}=z_i^* z_i$ for all $i=0,\ldots,n$. 
We task ourself to select the class of quantum weighted projective spaces for which the elements \eqref{eq:conj} 
generate the whole of $\A(\WP_q(\mv{\ell}))$.

If $n=1$, and for every $\mv{\ell}$, it is clearly true: 
in this case $z^{\mv{k}}$ has at most length $2$, and there are no invariant monomials of length $1$.
For arbitrary $n$, if $\ell_0=\ldots=\ell_n=1$ again \eqref{eq:conj} gives a set of generators, the matrix
elements of the defining projection of $\CP^n_q$, as shown in \cite{DL10}.

On the other hand, it is not difficult to find examples that do not satisfy this property.
If $\mv{\ell}=(1,2,3)$ the element $z_0z_1z_2^*$ is irreducible in $\A(\WP_q(\mv{\ell}))$
(it is not the product of invariant monomials of smaller length), hence one needs both monomials of length
$2$ and $3$ to generate the algebra. Similarly, if $\mv{\ell}=(1,2,3,7)$ the element $z_0^2z_1z_2z_3^*$ is irreducible
in $\A(\WP_q(\mv{\ell}))$, and one needs elements of length $2$, $3$ and $4$ to generate the algebra.

It turns out (cf.~Theorem~\ref{thm:11} below) that the set 
of elements $\xi_{i,j}$ as in \eqref{eq:conj} generate
the algebra $\A(\WP_q(\mv{\ell}))$ if and only if the weight vector $\mv{\ell}$ is as in Theorem~\ref{thm:6},
hence classically $\WP(\mv{\ell})\simeq\CP^n$. 
We need some preliminary lemmas. First of all, let us recall:

\begin{lemma}[B{\'e}zout's identity]\label{lemma:Bez}
Let $R$ be a principal ideal domain. For any $a,b\in R$ there exists $x,y\in R$ such that $\gcd(a,b)=ax+by$.
\end{lemma}

When $R=\Z$ and $a,b\geq 1$, it is an easy exercise to prove that one can always choose $x,y$ different
from zero and with opposite sign. So, as a corollary:

\begin{cor}\label{lemma:coprime}
For any three positive integers  $a,b,k$ there exists two non-zero integers $r,s$
with opposite sign such that $ar+bs=k\gcd(a,b)$. 
\end{cor}

\begin{lemma}\label{lemma:divide}
Let $\mv{\ell}$ be a coprime weight vector. There is a pairwise coprime weight vector $\mv{p}$ such that $\mv{\ell}=\mv{p}^{\,\sharp}$
if and only if $\ell_{i:j}=\ell_i/\mr{gcd}(\ell_i,\ell_j)$ divides $\ell_k$ for all $i\neq j$ and $j\neq k$.
\end{lemma}

\begin{proof}[Proof of ``$\,\Rightarrow$'']
If $\mv{\ell}=\mv{p}^{\,\sharp}$ with $\mv{p}$ pairwise coprime, one has $\ell_{i:j}=p_j$ for any $i\neq j$. 
Hence $\ell_{i:j}$ divides $\ell_k$ for all $k\neq j$.

\medskip

\noindent
\textit{Proof of ``$\,\Leftarrow$''.}
We prove it by induction. The statement is trivial if $n=0,1$. Let $n\geq 2$.

Define $p_0=\ell_{1:0}$, $p_1=\ell_{0:1}$ and $p'_2=\gcd(\ell_0,\ell_1)$. Note that $p_0$ and $p_1$ are coprime.
By construction $\ell_0=p_1p'_2$ and $\ell_1=p_0p'_2$. By hypothesis $p_0$ and $p_1$ divide $\ell_2,\ldots,\ell_n$.
Since they are coprime, for all $k\geq 2$,
there exists integers $r_k,s_k,\ell'_k\geq 1$ such that
$\ell_k=\ell'_kp_0^{r_k}p_1^{s_k}$ and $\ell'_k$ is not divisible by $p_0$ or $p_1$
(this is just the decomposition of an integer number in prime factors).
Since $\gcd(p_0,p'_2)$ divides all weights, it must be $1$.
Hence $p'_2$ is coprime to $p_0$, and similarly is coprime to $p_1$.
It follows that $\gcd(\ell_k,\ell_0)=p_1\gcd(\ell'_k,p'_2)$ and that
$$
\ell_{k:0}=\frac{\ell'_k}{\gcd(\ell'_k,p'_2)}\cdot p_0^{r_k}p_1^{s_k-1} \;.
$$
By hypothesis $\ell_{k:0}$ divides $\ell_1$, hence $p_0^{r_k-1}$ and $p_1^{s_k-1}$ divides $p'_2$. Since $p_0,p_1,p'_2$
are pairwise coprime, this implies $r_k=s_k=1$.

Since again $\ell_{k:0}=p_0\ell'_k/\gcd(\ell'_k,p'_2)$ divides $\ell_1$,
then $\ell'_k/\gcd(\ell'_k,p'_2)$ divides $p'_2$, which is only possible if
$p'_2$ is a multiple of $\ell'_k$. Since $p'_2$ is coprime to $p_0$ and $p_1$,
$\ell'_k$ is coprime to $p_0$ and $p_1$ for all $k\geq 2$. Furthermore,
$\gcd(\ell'_2,\ldots,\ell'_n)$ divides $\ell_i$ for all $i=0,\ldots,n$,
but $\mv{\ell}$ is coprime, so $\mv{\ell}'$ must be coprime too.

The next step is to show that the weight vector $\mv{\ell}'$ satisfies the condition of Lemma~\ref{lemma:divide}.
Let $i\neq j$ and $i,j\geq 2$. Being $\ell'_i$ coprime to $p_0$ and $p_1$, so is $\ell'_{i:j}$.
But $\ell'_{i:j}=\ell'_i/\gcd(\ell'_i,\ell'_j)=\ell_{i:j}$ and $\ell_{i:j}$ divides $\ell_k$;
hence $\ell'_{i:j}$ divides $\ell'_k$ (for all $k\neq j$, $k\geq 2$).

Now we use the inductive hypothesis, $\mv{\ell}'=(\mv{p}^{\geq 2})^{\,\sharp}$ for some $\mv{p}^{\geq 2}=(p_2,p_3,\ldots,p_n)$
pairwise coprime. Since $\ell'_2,\ldots,\ell'_n$ are coprime to $p_0$ and $p_1$, the vector $\mv{p}:=(p_0,\ldots,p_n)$ is pairwise coprime.
Moreover, $\ell_i=\prod_{j\neq i}p_j$ for all $i\geq 2$. It remains to prove that $p'_2=p_2p_3\ldots p_n$.

For all $i\neq j$ with $i,j\geq 2$, $\ell_{i:j}=p_j$ divides $\ell_0$. Hence $\ell_0=rp_1p_2^{k_2}p_3^{k_3}\ldots p_n^{k_n}$
for integers $r,k_2,\ldots,k_n\geq 1$. Since $\ell_{0:2}=rp_2^{k_2}p_3^{k_3-1}\ldots p_n^{k_n-1}$ divides $\ell_3$,
$p_3^{k_3-1}$ divides $\ell_3$, and then $k_3=1$. Similarly one shows that $k_i=1$ for all $i\geq 2$. So,
$p'_2=rp_2p_3\ldots p_n$. From the above expression for $\ell_{0:2}$ it follows that $r$ divides $\ell_3$,
and similarly $\ell_i$ for all $i\geq 2$. On the other hand, $r$ divides $p'_2$ and then $\ell_0$ and $\ell_1$.
Since $\mv{\ell}$ is a coprime vector, it must be $r=1$.
\end{proof}

\begin{thm}\label{thm:11}
Let $\mv{\ell}$ be a coprime weight vector. Then $\A(\WP_q(\mv{\ell}))$ is generated by the elements $\xi_{i,j}$ if and only
if there is a pairwise coprime weight vector $\mv{p}$ such that $\mv{\ell}=\mv{p}^{\,\sharp}$.
\end{thm}
\begin{proof}[Proof of ``$\,\Rightarrow$'']
Since the statement is trivial for $n=1$, we assume $n\geq 2$.
We assume the thesis is false, and prove that this contradicts the hypothesis, i.e. we
assume there is no $\mv{p}$ such that $\mv{\ell}=\mv{p}^{\,\sharp}$, and prove the existence of
an irreducible monomial with length $3$.

By Lemma~\ref{lemma:divide}, there exists indices $i,j,k$ with $i\neq j$ and $j\neq k$ such that
$\ell_{i:j}$ does not divide $\ell_k$. It must be $k\neq i$, since $\ell_{i:j}$ always divides $\ell_i$.
To simplify the notations, we may assume $i=0$, $j=1$, $k=2$, the general case being the same.

Using Corollary~\ref{lemma:coprime} for $a=\ell_0$, $b=\ell_2$ and $k=\ell_1$ one can find non-zero integers $r,s$
with opposite sign such that $\ell_0 r+\ell_2 s=\ell_1\gcd(\ell_0,\ell_2)$. If $r$ is positive, the length $3$ monomial
\begin{equation}\label{eq:notgen}
(z_0^*)^{|r|}z_1^{\gcd(\ell_0,\ell_2)}z_2^{|s|}
\end{equation}
belongs to $\A(\WP_q(\mv{\ell}))$.
By hypothesis $\ell_{0:1}$ does not divide $\ell_2$, thus it does not divide $\gcd(\ell_0,\ell_2)$.
On the other hand, if an invariant element of the form $(z_0^*)^{k_0}z_1^{k_1}z_2^{k_2}$ (with positive $k_0,k_1,k_2$)
is generated by element $\xi_{i,j}$ as in \eqref{eq:conj}, then $\ell_{0:1}$ divides $k_1$ (since $\xi_{0,1}^\alpha\xi_{0,2}^\beta=(z_0^*)^{\alpha\ell_{1:0}+\beta\ell_{2:0}}
z_1^{\alpha\ell_{0:1}}z_2^{\beta\ell_{0:2}}$ for all $\alpha,\beta\geq 1$). This proves that \eqref{eq:notgen} is not a product of elements in \eqref{eq:conj}. For $r$ negative, we repeat the proof with the monomial $z_0^{|r|}z_1^{\gcd(\ell_0,\ell_2)}(z_2^*)^{|s|}$.

\medskip

\noindent
\textit{Proof of ``$\,\Leftarrow$''.}
Let $\mv{\ell}=\mv{p}^{\,\sharp}$ with $\mv{p}$ pairwise coprime. If $z^{\mv{k}}$ is $\U(1)$-invariant, one has that 
$$
-k_0\ell_0=k_1\ell_1+\ldots+k_n\ell_n \;.
$$
Since $p_0$ divides every weight in the right hand side, it divides $k_0\ell_0$;
since it does not divide $\ell_0$, it has to divide $k_0$.
Similarly $p_i$ divides $k_i$ for all $i=1,\ldots,n$.
After reparametrization, any invariant monomial $z^{\mv{k}}$ in Lemma~\ref{lemma:2} 
is of the form $\zeta^{\mv{k}}:=\zeta_0^{k_0}\zeta_1^{k_1}\ldots \zeta_n^{k_n}$ with $\zeta_i:=z_i^{p_i}$.
Note that for all $i\neq j$, being $\ell_{i:j}=p_j$,
elements in \eqref{eq:conj} are given by 
$$
\xi_{i,j}=\zeta_i^*\zeta_j \, .
$$
Named $p:=\prod_{j=0}^np_j$,
since $\alpha_t(\zeta_i)=t^p\zeta_i$ for all $i=0,\ldots,n$ ($p_i\ell_i=p$ for all $i$), an invariant monomial $\zeta^{\mv{k}}$
contains the same number of $\zeta_i$'s and of $\zeta_i^*$'s (each counted with multiplicities).
Using the relation \eqref{eq:defrelB} the factors can be reordered so that $\zeta_i$'s and $\zeta_i^*$'s
are alternating, i.e. we can write an invariant $\zeta^{\mv{k}}$ as a product of elements $\xi_{i,j}=\zeta_i^*\zeta_j$.
\end{proof}

For the classes of spaces in Theorem~\ref{thm:11}, it is clear from the proof that $\A(\WP_q(\mv{\ell}))$ is a
\mbox{$*$-subalgebra} of the algebra generated by the elements:
\begin{equation}\label{eq:genLq}
x_i:=z_iz_i^* \;,\quad
\zeta_i:=z_i^{p_i} \;,\quad
\zeta_i^* \;, \qquad \textup{for} \quad i=0,\ldots,n \, .
\end{equation}
If $\mv{\ell}=\mv{p}^{\,\sharp}$ with $\mv{p}$ pairwise coprime, 
this algebra is the lens space algebra $\A(L_q(p;\mv{\ell}))$ defined in \eqref{eq:Lq}, 
for the action of the cyclic group $\Z_p$, with parameter $p$ now given by $p:=p_0p_1\ldots p_n$.

\begin{thm}\label{thm:13}
Let $\mv{\ell}=\mv{p}^{\,\sharp}$ with $\mv{p}$ pairwise coprime, and let $p:=p_0p_1\ldots p_n$. 
Then the algebra $\A(L_q(p;\mv{\ell}))$ is generated by the elements \eqref{eq:genLq}.
\end{thm}
\begin{proof}
By arguing as in the proof of Lemma~\ref{lemma:2}, clearly $\A(L_q(p;\mv{\ell}))$ is generated by the elements $x_i=z_iz_i^*$ 
and by monomials $z^{\mv{k}}$ that are $\Z_p$-invariant, which happens if and only if $\mv{k}\cdot\mv{\ell}\in p\Z$.
Using $\mv{\ell}=\mv{p}^{\,\sharp}$ the invariance condition becomes
$$
\sum_{i=0}^nk_i\prod_{j\neq i}p_j\in (p_0\ldots p_n) \Z.
$$
Every summand in the equation above besides the $0$-th is divisible by $p_0$, hence $k_0\ell_0$ must be divisible by $p_0$,
i.e. $k_0$ is divisible by $p_0$. Similarly $k_i\in p_i\Z$ for all $i$. Thus $z^{\mv{k}}=\zeta_0^{h_0}\zeta_1^{h_1}\ldots\zeta_n^{h_n}$
where $h_i=k_i/p_i$. Since each $\zeta_i$ is $\Z_p$-invariant, such a monomial clearly belongs to $\A(L_q(p;\mv{\ell}))$,
thus concluding the proof.
\end{proof}

We close this section by computing the relations among the generating elements \eqref{eq:genLq}.

\begin{prop}\label{prop:relLens}
The elements $x_i$, $\zeta_i$, $\zeta_i^*$ satisfy the commutation relations
\begin{subequations}\label{eq:relLens}
\begin{align}
x_ix_j &= x_jx_i && \textup{for all} \quad i,j \;, \label{eq:relLensA} \\
x_i\zeta_j &=\zeta_jx_i && \textup{for all} \quad 0\leq i<j\leq n \;, \label{eq:relLensB} \\
x_j\zeta_i &=q^{2p_i}\zeta_ix_j && \textup{for all} \quad 0\leq i<j\leq n \;, \label{eq:relLensC} \\
\zeta_i\zeta_j &=q^{-p_ip_j}\zeta_j\zeta_i && \textup{for all} \quad 0\leq i<j\leq n \;, \label{eq:relLensD} \\
\zeta_i^*\zeta_j &=q^{p_ip_j}\zeta_j\zeta_i^* && \textup{for all} \quad i\neq j \;, \label{eq:relLensE} \\
[x_i,\zeta_i] &=(1-q^{2p_i})\,\zeta_i\sum\nolimits_{j>i}x_j
    \hspace{-1cm} && \textup{for all} \quad i=0,\ldots,n \;, \label{eq:relLensF} \\
\intertext{together with the relations:}
x_0 & +x_1+\ldots +x_n =1 \;, \label{eq:relLensG} \\
\zeta_i\zeta_i^* &=\prod_{k=0}^{p_i-1}\Big\{x_i+(1-q^{-2k})\sum\nolimits_{j>i}x_j\Big\}
    && \textup{for all} \quad i=0,\ldots,n \;, \label{eq:relLensH} \\
\zeta_i^*\zeta_i &=\prod_{k=1}^{p_i}\Big\{x_i+(1-q^{2k})\sum\nolimits_{j>i}x_j\Big\}
    && \textup{for all} \quad i=0,\ldots,n \;. \label{eq:relLensL} 
\end{align}
\end{subequations}
It is understood that an empty sum is $0$.
\end{prop}

\begin{proof}
The relations from \eqref{eq:relLensA} to \eqref{eq:relLensE} are easy to derive. We move to the next.
Let
$$
X_i:=\sum\nolimits_{j>i}x_j
$$
and note that $X_iz_i=q^2z_iX_i$.
It follows by induction on $k\geq 1$ that
\begin{equation}\label{eq:zstarzi}
[z_i^*,z_i^k] = (1-q^{2k})\,z_i^{k-1}X_i \;.
\end{equation}
For $k=1$ this is just \eqref{eq:defrelC}, and from the algebraic identity
\begin{align*}
[z_i^*,z_i^{k+1}] &=
[z_i^*,z_i^k]z_i+z_i^k[z_i^*,z_i]
=(1-q^{2k})z_i^{k-1}X_iz_i+(1-q^2)z_i^kX_i
\\
&=(1-q^{2k})q^2z_i^kX_i+(1-q^2)z_i^kX_i
=(1-q^{2k+2})z_i^kX_i
\end{align*}
the inductive step follows.
From \eqref{eq:zstarzi} and $[x_i,z_i^{k+1}]=z_i[z_i^*,z_i^{k+1}]$
we get \eqref{eq:relLensF} for $k=p_i$.

For $k\geq 1$, it follows from \eqref{eq:zstarzi} that
$$
z_i^{k+1}(z_i^*)^{k+1}=
z_iz_i^*z_i^k(z_i^*)^k-z_i[z_i^*,z_i^k](z_i^*)^k=
\big\{x_i+(1-q^{-2k})X_i\big\}z_i^k(z_i^*)^k \;.
$$
That is
\begin{equation}\label{eq:Yk}
Y_i(k)=\big\{x_i+(1-q^{-2(k-1)})X_i\big\}Y_i(k-1) \;,
\end{equation}
where $Y_i(k):=z_i^k(z_i^*)^k$ if $k\geq 1$ and $Y_i(0)=1$.
By iterated use of \eqref{eq:Yk} we find:
$$
Y_i(k)=\prod_{j=0}^{k-1}\big\{x_i+(1-q^{-2j})X_i\big\} \;.
$$
This gives \eqref{eq:relLensH} when $k=p_i$.
Note that the order in the product does not matter,
since it follows from \eqref{eq:relLensA} that $x_i$ and $X_i$ commute.

Similarly, using the conjugate of \eqref{eq:zstarzi}:
$$
-[z_i,(z_i^*)^k] = (1-q^{2k})X_i(z_i^*)^{k-1} \;,
$$
and
$$
(z_i^*)^{k+1}z_i^{k+1}=z_i^*z_i(z_i^*)^kz_i^k-z_i^*[z_i,(z_i^*)^k]z_i^k \;,
$$
for $Z_i(k):=(z_i^*)^kz_i^k$, we find
\begin{align*}
Z_i(k+1)&=\Big\{z_i^*z_i+(1-q^{2k})q^2X_i\Big\}Z_i(k) \\
&=\Big\{x_i+(1-q^2)X_i+(1-q^{2k})q^2X_i\Big\}Z_i(k)
=\Big\{x_i+(1-q^{2k+2})X_i\Big\}Z_i(k) \;.
\end{align*}
By iterated use of this equation we arrive at:
$$
Z_i(k)=\prod_{j=1}^k\big\{x_i+(1-q^{2j})X_i\big\} \;.
$$
This implies \eqref{eq:relLensL}.
Last relation \eqref{eq:relLensG} is simply \eqref{eq:defrelE}.
\end{proof}
From \eqref{eq:relLensH} and \eqref{eq:relLensL} one also computes for all $i=0,\ldots,n$, the commutator:
$$
[\zeta_i^*,\zeta_i]=(q-q^{-1})\sum_{k=0}^{p_0}\sqbn{p_0}{k}[\hspace{1pt}p_0k]_q
\left(-q\sum\nolimits_{j\geq i}x_j\right)^k
\left(\sum\nolimits_{j\geq i+1}x_j\right)^{p_0-k} \;,
$$
where $\sqbn{p_0}{k}$ is the $q$-binomial (cf.~\eqref{eq:qbinom}).
In the equations \eqref{eq:relLens} we separated a first group of relations, reducing to the property of the algebra
being commutative when $q=1$, and a second group which is a deformation of the algebraic equations
defining the lens space $L(p;\mv{\ell})$.

\begin{rem}
For $n=1$, 
the algebra $\A(L_q(p;1,p))$ is isomorphic to the abstract unital $*$-algebra 
generated by elements as in \eqref{eq:genLq} with relations as in
Proposition~\ref{prop:relLens} (essentially meaning there are no additional relations among the generators). While it ought to be possible to   establish an analogous statement for general $n$ and any weight $\mv{\ell}$ as in Theorem \ref{thm:13},
such a result is not needed in the following.
\end{rem}


\section{Irreducible representations}\label{sec:irreps}

Irreducible representation of quantum spheres were constructed in \cite{HL04}. From these, by restriction one gets 
irreducible representations of quantum lens and weighted projective spaces. They will be used in the next section to construct
Fredholm modules.

Denote by $\ket{\vec{k}}$ the canonical orthonormal basis of $\ell^2(\N^n)$, where $\vec{k}=(k_1,\ldots,k_n)\in\N^n$,
and by $\vec{e}_i$ the vector with $i$-th component equal to $1$ and all the others equal to zero (for $i=1,\ldots,n$).
A faithful representation of $\A(S^{2n+1}_q)$ on $\ell^2(\N^n)$ is given on generators by
\begin{align*}
z_i\ket{\vec{k}} &=q^{k_1+\ldots+k_i}\sqrt{1-q^{2(k_{i+1}+1)}}\ket{\vec{k}+\vec{e}_{i+1}} \;, \qquad \textup{for} \;\; 0\leq i<n \, , \\
z_n\ket{\vec{k}} &=q^{k_1+\ldots+k_n}\ket{\vec{k}} \, ,
\end{align*}
where we omit the representation symbol.
This is the representation $\psi_1^{(2n+1)}$ of \cite{HL04}, modulo a renaming of the generators and a redefinition of the parameters.

Assume now that the hypothesis of Theorem~\ref{thm:13}, are satisfied, that is let $\mv{\ell}=\mv{p}^{\,\sharp}$ be a weight vector, with $\mv{p}$ pairwise coprime, and let $p:=p_0p_1\ldots p_n$. On the generators \eqref{eq:genLq} of $\A(L_q(p;\mv{\ell}))$ the representation above gives:
\begin{align*}
x_i\ket{\vec{k}} &=q^{2(k_1+\ldots+k_i)}\big(1-q^{2k_{i+1}}\big)\ket{\vec{k}} \;, & \textup{for} \;\; 0\leq i<n, \\
x_n\ket{\vec{k}} &=q^{2(k_1+\ldots+k_n)}\ket{\vec{k}} \;, \\
\zeta_i\ket{\vec{k}} &=q^{p_i(k_1+\ldots+k_i)}\sqrt{ \tbrace{k_{i+1}+p_i}{k_{i+1}} }\ket{\vec{k}+p_i\vec{e}_{i+1}} \;, & \textup{for} \;\; 0\leq i<n, \\
\zeta_n\ket{\vec{k}} &=q^{p_n(k_1+\ldots+k_n)}\ket{\vec{k}} \;,
\end{align*}
where, for $0\leq k<m$,
$$
\dbrace{m}{k} := (1 - q^{2k+2})(1 - q^{2k+4})\ldots (1 - q^{2m})
$$
is just a shorthand notation for the $q$-shifted factorial $(q^{2k+2};q^2)_{m-k-1}$.

This representation breaks into irreducible components for $\A(L_q(p;\mv{\ell}))$. To see this,
we relabel the basis vectors as follows. For all $i=1,\ldots,n$ let
$$
k_i=p_{i-1}(m_i-m_{i-1})+r_{i-1}\, ,
$$
where $\vec{m}=(m_1,\ldots,m_n)\in\N^n$ satisfies $0\leq m_1\leq m_2\leq\ldots\leq m_n$, we set $m_0:=0$, and $r_i\in\{0,\ldots,p_i-1\}$ are the remainders.
The inverse transformation is then
\begin{equation}\label{eq:inverse}
m_i=\sum_{j=1}^i\frac{k_j-r_{j-1}}{p_{j-1}} \;.
\end{equation}
Basis vectors will be renamed accordingly $\ket{\vec{m};\vv{r}}$ and the representation breaks into the
irreducible sub-representations of $\A(L_q(p;\mv{\ell}))$ given in the next proposition.

\begin{prop}\label{prop:A}
Fix a vector $\vv{r}=(r_0,\ldots,r_{n-1}) \in\N^n$ with constraints on the components 
\begin{gather*}
\quad 0 \leq r_i<p_i\;, \quad \textup{for} \,  \, \, i=0,\ldots,n-1 \, .
\end{gather*}
Let $\HH_{\vv{r}}$ be the Hilbert space with orthonormal basis $\ket{\vec{m};\vv{r}}$ and \mbox{$\vec{m}=(m_1,\ldots,m_n)$}
such that 
\begin{gather*}
0\leq m_1\leq m_2\leq\ldots\leq m_n \, .
\end{gather*}
An irreducible representation of the lens algebra $\A(L_q(p;\mv{\ell}))$ is given on generators
by\vspace{3pt}
\begin{align*}
x_i\ket{\vec{m};\vv{r}} &=
               q^{2\sum_{j=0}^{i-1}r_j}q^{2\sum_{j=1}^ip_{j-1}(m_j-m_{j-1})}
               \big(1-q^{2r_i}q^{2p_i(m_{i+1}-m_i)}\big)\ket{\vec{m};\vv{r}} \;, \\[2pt]
\zeta_i\ket{\vec{m};\vv{r}} &=
               q^{p_i\sum_{j=0}^{i-1}r_j}q^{\sum_{j=1}^ip_{j-1}(m_j-m_{j-1})}
               \sqrt{ \dbrace{ p_i(m_{i+1}-m_i+1) }{ p_i(m_{i+1}-m_i) } }\ket{\vec{m}+\vec{e}_{in};\vv{r}} \;, \\[-3pt]
\intertext{for all $i=0,\ldots,n-1$ and\vspace{-3pt}}
x_n\ket{\vec{m};\vv{r}} &=
               q^{2\sum_{j=0}^{n-1}r_j}q^{2\sum_{j=1}^np_{j-1}(m_j-m_{j-1})}
               \ket{\vec{m};\vv{r}} \;, \\[2pt]
\zeta_n\ket{\vec{m};\vv{r}} &=
               q^{p_n\sum_{j=0}^{n-1}r_j}q^{\sum_{j=1}^np_{j-1}(m_j-m_{j-1})}
               \ket{\vec{m};\vv{r}} \;,
\end{align*}
where $\vec{e}_{in}:=(\, \stackrel{i\;\mr{times}}{\overbrace{0,0,\ldots,0}}\,,\stackrel{n-i\;\mr{times}}{\overbrace{1,1,\ldots,1}}\, )$,
for all $0\leq i<n$, with $m_0:=0$.
\end{prop}

\section{Fredholm modules}\label{sec:Khom}

Here we present some basic Fredholm modules for the algebra $\A(\WP_q(\mv{\ell}))$ constructed using faithful
representations. Additional Fredholm modules can then be obtained by iterated pullbacks from `lower dimensions'. 
Indeed, the epimorphism $\A(S^{2n+1}_q)\twoheadrightarrow\A(S^{2n-1}_q)$ given by $z_n\mapsto 0$
induces an epimorphism $\A(\WP_q(\ell_0,\ldots,\ell_n))\twoheadrightarrow\A(\WP_q(\ell_0,\ldots,\ell_{n-1}))$.

The building block representations are the ones described in \S\ref{sec:irreps}. As done there,
we assume that the hypothesis of Theorem~\ref{thm:13} are satisfied: that is $\mv{\ell}=\mv{p}^{\,\sharp}$ with $\mv{p}$ pairwise coprime.
The next definition is the analogue of \cite[Def.~1]{DL10}. Through the whole section, we assume
we fixed a sequence of integers $\vv{r}=(r_0,\ldots,r_{n-1})$ satisfying
\begin{equation}\label{eq:ri}
0\leq r_i<p_i \;,\qquad \textup{for} \quad i=0,\ldots,n-1 \, .
\end{equation}

\begin{df}\label{def:15}
Let $\HH_n:=\ell^2(\N^n)$, with orthonormal basis $\ket{\vec{m}}$, $\vec{m}=(m_1,\ldots,m_n)\in\N^n$.
For $0\leq k\leq n$ let $\mc{V}^n_k\subset\HH_n$ be the linear span of basis vectors 
$\ket{\vec{m}}$ satisfying the constraints
\begin{equation}\label{eq:mconstr}
0\leq m_1\leq m_2\leq\ldots\leq m_k\;,\qquad\quad m_{k+1}>m_{k+2}>\ldots>m_n\geq 0 \;,
\end{equation}
the former condition being empty if $k=0$, and the latter one being empty if $k=n$.
For every $0\leq k\leq n$, a representation \mbox{$\pi^{(n)}_k:\A(L_q(p;\mv{\ell}))\to\B(\HH_n)$}
is defined as follows (all the representations are on the same Hilbert space). 
Firstly, for $0\leq i<k\leq n$, denote by $\vec{e}_{ik}\in\{0,1\}^n$ the array
$$
\vec{e}_{ik}:= 
(\,\,\stackrel{i\;\mr{times}}{\overbrace{0,0,\ldots,0}}\,,
\stackrel{k-i\;\mr{times}}{\overbrace{1,1,\ldots,1}}\,,
\stackrel{n-k\;\mr{times}}{\overbrace{0,0,\ldots,0}}) \;.  
$$
Then, we set $\pi^{(n)}_k(x_i)=\pi^{(n)}_k(\zeta_i)=0$ if $i>k$, while the remaining generators are given by: \vspace{5pt}
\begin{align*}
\pi^{(n)}_k(x_i)\ket{\vec{m}} &=
               q^{2\sum_{j=0}^{i-1}r_j}q^{2\sum_{j=1}^ip_{j-1}(m_j-m_{j-1})}
               \big(1-q^{2r_i}q^{2p_i(m_{i+1}-m_i)}\big)\ket{\vec{m}} \, , \quad \textup{for} \;0\leq i<k, \\[5pt]
\pi^{(n)}_k(x_k)\ket{\vec{m}} &=
               q^{2\sum_{j=0}^{k-1}r_j}q^{2\sum_{j=1}^kp_{j-1}(m_j-m_{j-1})}
               \ket{\vec{m}} \;, \\
\pi^{(n)}_k(\zeta_i)\ket{\vec{m}} &=
               q^{p_i\sum_{j=0}^{i-1}r_j}q^{\sum_{j=1}^ip_{j-1}(m_j-m_{j-1})}
               \sqrt{ \dbrace{ p_i(m_{i+1}-m_i+1) }{ p_i(m_{i+1}-m_i) } }\ket{\vec{m}+\vec{e}_{ik}} \;, \\
                & \hspace{11cm} \textup{for} \;0\leq i<k, \\[-5pt]
\pi^{(n)}_k(\zeta_k)\ket{\vec{m}} &=
               q^{p_k\sum_{j=0}^{k-1}r_j}q^{\sum_{j=1}^kp_{j-1}(m_j-m_{j-1})}
               \ket{\vec{m}} \;,
\end{align*}
\rule{0pt}{13pt}(with $m_0:=0$) on the subspace $\mc{V}^n_k$ and they are zero on the orthogonal subspace.
\end{df}

The representation $\pi^{(n)}_k$, when restricted to $\mc{V}^n_k$, is the direct sum
of several copies of the irreducible representation for $\A(L_q(p_0\ldots p_k;\ell_0,\ldots,\ell_k))$ 
given in Proposition~\ref{prop:A},
and pulled back to $\A(L_q(p;\ell_0,\ldots,\ell_n))$:
one copy of the representation for each value of the additional labels $m_{k+1},\ldots,m_n$.

Lemma~2 of \cite{DL10} still holds (and we do not repeat the proof here):

\begin{lemma}\label{le:pl} The spaces $\mathcal{V}^n_k$ are such that $\mathcal{V}^n_j\perp\mathcal{V}^n_k$ if $|j-k|>1$, while
$\mathcal{V}^n_{k-1}\cap\mathcal{V}^n_k$, for $1\leq k\leq n$, is the span of vectors $\ket{\vec{m}}$ satisfying:
\begin{equation}\label{eq:capconstr}
0\leq m_1\leq m_2\leq\ldots\leq m_k\;, \qquad\quad m_k>m_{k+1}>\ldots>m_n\geq 0 \;.
\end{equation}
\end{lemma}

As a consequence $\pi^{(n)}_j(a)\pi^{(n)}_k(b)=0$ for all $a,b\in\A(L_q(p;\mv{\ell}))$, if $|j-k|>1$,
and the maps $\pi_\pm^{(n)}:\A(L_q(p;\mv{\ell}))\to\B(\HH_n)$ defined by
\begin{equation}\label{irreps+-}
\pi_+^{(n)}(a):=\sum_{\substack{0\leq k\leq n \\[1pt] k\;\mr{even}}}\pi^{(n)}_k(a) \;,\qquad
\pi_-^{(n)}(a):=\sum_{\substack{0\leq k\leq n \\[1pt] k\;\mr{odd}}}\pi^{(n)}_k(a) \;,
\end{equation}
are representations of the algebra $\A(L_q(p;\mv{\ell}))$.
We then generalize \cite[Prop.~3]{DL10}.

\begin{prop}\label{prop}
For all $a\in\A(\WP_q(\mv{\ell}))$, the operator $\pi_+^{(n)}(a)-\pi_-^{(n)}(a)$ is of trace class on $\HH_n$;
furthermore, the trace is given by a series which --- as a function of $q$ --- is absolutely convergent in the open interval $0<q<1$.
\end{prop}

\begin{proof}
The space $\HH_n$ is the orthogonal direct sum of $\mc{V}^n_{k-1}\cap\mc{V}^n_k$, for all $1\leq k\leq n$, plus the joint kernel of all the representations involved.
From Lemma~\ref{le:pl}, on $\mc{V}^n_{k-1}\cap\mc{V}^n_k$ only the representations $\pi^{(n)}_{k-1}$ and $\pi^{(n)}_k$ are different from zero 
(accordingly to the parity of $k$ one contributes to $\pi_+^{(n)}$ and the other to $\pi_-^{(n)}$). 
It then suffices to prove that $\pi^{(n)}_{k-1}(a)-\pi^{(n)}_k(a)$ is of trace class, and that the trace is absolutely convergent for any $0<q<1$.
Moreover, it is enough to show this for $a=\xi_{i,j}$ a generator of $\A(\WP_q(\mv{\ell}))$ as given in \eqref{eq:conj}, with $0\leq i\leq j\leq n$.
Remember that $\xi_{i,i} =z_i^*z_i$ 
(these can be replaced by the generators $x_i=z_iz_i^*$), and that $\xi_{i,j}=\zeta_i^*\zeta_j$ for all $i\neq j$.

The explicit expressions in Definition~\ref{def:15}, yields that both $\pi^{(n)}_{k-1}(\xi_{i,j})$ and $\pi^{(n)}_k(\xi_{i,j})$ vanish if $j>k$. 
For $j=k$, $\pi^{(n)}_{k-1}(\xi_{i,k})$ vanishes and $\pi^{(n)}_k(\xi_{i,k})$ has matrix coefficients bounded by 
\begin{equation}\label{eq:bound}
q^{\sum_{j=1}^kp_{j-1}(m_j-m_{j-1})} \;.
\end{equation}
For $j=k-1$, one uses the inequality $|1-\sqrt{1-x^2}|\leq x$ (which is valid for $0\leq x\leq 1$) to prove that $\pi^{(n)}_{k-1}(\xi_{i,j})-\pi^{(n)}_k(\xi_{i,j})$ still has matrix coefficients bounded by \eqref{eq:bound}. For $0\leq i\leq j\leq k-2$, the operators $\pi^{(n)}_{k-1}(\xi_{i,j})$ and $\pi^{(n)}_k(\xi_{i,j})$ coincide on $\mc{V}^n_{k-1}\cap\mc{V}^n_k$.

Since $m_j-m_{j-1}\geq 0$ for $1\leq j\leq k$ and $p_{j-1}\geq 1$, $q^{p_{j-1}(m_j-m_{j-1})}\leq q^{m_j-m_{j-1}}$ and the coefficient in \eqref{eq:bound}
is bounded by $q^{m_k}$. The observation that the series
$$
\sum_{\vec{m}\;\textrm{satisfying (\ref{eq:capconstr})}}q^{m_k}
=\sum_{m_k=n-k}^\infty\binom{m_k+k-1}{k-1}\binom{m_k}{n-k}q^{m_k}
$$
is absolutely convergent for $0<q<1$ concludes the proof.
\end{proof}

As a consequence of Proposition~\ref{prop}, using the direct sum of the representation $\pi_+$ and $\pi_-$ we can construct a 
Fredholm module in a standard manner.

Let us introduce the label $\vv{r}$ for book-keeping.
Let $\pi_{n,\vv{r}}^\pm$ be the representations in \eqref{irreps+-} and
$\HH_{n,\vv{r}}^\pm$ two copies of the underlying Hilbert space previously denoted $\HH_n$.
Let
$$
\pi_{n,\vv{r}}=\pi_{n,\vv{r}}^+\oplus\pi_{n,\vv{r}}^- \;,\qquad
\HH_{n,\vv{r}}=\HH_{n,\vv{r}}^+\oplus\HH_{n,\vv{r}}^- \;,
$$
let $\gamma_{n,\vv{r}}$ be the obvious grading on $\HH_{n,\vv{r}}$ and $F_{n,\vv{r}}$
the flip operator: $F_{n,\vv{r}}(v\oplus w)=w\oplus v$. Then, for $\mv{\ell}=\mv{p}^{\,\sharp}$ with $\mv{p}$ a pairwise coprime weight vector, the datum
\begin{equation}\label{eq:fmn}
\big(\,
\A(\WP_q(\mv{\ell}))
\,,\,
\HH_{n,\vv{r}}
\,,\,
\pi_{n,\vv{r}}
\,,\,
F_{n,\vv{r}}
\,,\,
\gamma_{n,\vv{r}}
\,\big)
\end{equation}
is a $1$-summable even Fredholm module.
Due to \eqref{eq:ri}, the number of such Fredholm modules is the number of possible values of the label $\vv{r}$, that is $p_0p_1\ldots p_{n-1}$.

Additional Fredholm modules are obtained by pullback, applying the same construction to $\A(\WP_q(\ell_0,\ldots,\ell_k))$,
for all $k=1,\ldots,n-1$. Note that $\A(\WP_q(\ell_0,\ldots,\ell_k))$ coincides with the algebra having coprime weight vector
$(p_0,\ldots,p_k)^{\,\sharp}$ (since $\gcd(\ell_0,\ldots,\ell_k)=p_{k+1}\ldots p_n$).

A final Fredholm module is the pullback of the canonical non-trivial Fredholm
module of $\C$, given on $\C\oplus\C$ by the representation $c\mapsto c\oplus 0$ and by
the usual $\gamma$ and $F$ operators.

The number of Fredholm modules we get in this way for the algebra $\A(\WP_q(\mv{\ell}))$ is then 
\begin{equation}\label{eq:number}
1+\sum_{k=1}^np_0p_1\ldots p_{k-1} \;.
\end{equation}
For $k\geq 1$ and $\vv{r}=(r_0,\ldots,r_{k-1})$, with $0\leq r_i<p_i$ for $0\leq i<k$,
we will denote the class of the Fredholm module 
$(\HH_{k,\vv{r}},\pi_{k,\vv{r}},F_{k,\vv{r}},\gamma_{k,\vv{r}})$, pulled-back to
$\A(\WP_q(\mv{\ell}))$, by $\mc{F}_{k,\vv{r}}$. The class of last Fredholm module is
denoted by $\mc{F}_{0,\vv{r}}$, with the convention that $\vv{r}=\emptyset$ in this case.

\section{Spectral triples}\label{sec:otto}
Let 
$
\big(\,
\A(\WP_q(\mv{\ell}))
\,,\,
\HH_{n,\vv{r}}
\,,\,
\pi_{n,\vv{r}}
\,,\,
F_{n,\vv{r}}
\,,\,
\gamma_{n,\vv{r}}
\,\big)
$
be the (irreducible) Fredholm module in \eqref{eq:fmn}. Recall that 
$\HH_{n,\vv{r}}=\HH_{n,\vv{r}}^+\oplus\HH_{n,\vv{r}}^-\simeq\ell^2(\N^n)\otimes\C^2$ and
that $\pi_{n,\vv{r}}=\pi_{n,\vv{r}}^+\oplus\pi_{n,\vv{r}}^-$,
where each summand is the sum of several orthogonal representations
$\pi^{(n)}_k$ (cf.~equation \eqref{irreps+-} and Definition~\ref{def:15}).
Let us denote $\|\vec{m}\|_1:=m_1+\ldots+m_n$ and let 
$$
D_{n,\vv{r}}:=|D_{n,\vv{r}}|F_{n,\vv{r}}
$$
with $|D_{n,\vv{r}}|$ the selfadjoint operator on $\ell^2(\N^n)$ defined by:  
\begin{equation}\label{eq:lambdaD}
|D_{n,\vv{r}}|\ket{\vec{m}}=\|\vec{m}\|_1\ket{\vec{m}}\;, \quad \forall\;\vec{m}\in\N^n\;.
\end{equation}
Given $\vec{k}\in\N^n$ and a bounded function $c:\N^n\to\C$, we call the bounded
operator
\begin{equation}\label{eq:WS}
S(\vec{k},c):\ket{\vec{m}}\mapsto c(\vec{m})\ket{\vec{m}+\vec{k}}
\end{equation}
and its adjoint \emph{weighted shifts}. Weighted shifts are eigenvectors of the derivation
$[|D_{n,\vv{r}}|,\,.\,]$, that is to say $[|D_{n,\vv{r}}|,S(\vec{k},c)]=\|\vec{k}\|_1S(\vec{k},c)$, and similarly for the adjoint.

\begin{prop}\label{prop:st}
The datum 
$
\big(\,
\A(\WP_q(\mv{\ell}))
\,,\,
\HH_{n,\vv{r}}
\,,\,
\pi_{n,\vv{r}}
\,,\,
D_{n,\vv{r}}
\,,\,
\gamma_{n,\vv{r}}
\,\big)
$
is an even spectral triple of metric dimension $n$.
\end{prop}
\begin{proof}
We have to show that $[D_{n,\vv{r}},\pi_{n,\vv{r}}(a)]$ is bounded for any generator 
$a$ of $\A(\WP_q(\mv{\ell}))$ (and thus for every element of the algebra, due to the Leibniz rule),
and that $|D_{n,\vv{r}}|^{-k}$ is traceclass (outside $\ker|D_{n,\vv{r}}|$) for every $k>n$. 
Observe that:
$$
[D_{n,\vv{r}},\pi_{n,\vv{r}}(a)]=[|D_{n,\vv{r}}|,\pi_{n,\vv{r}}(a)]F_{n,\vv{r}}+|D_{n,\vv{r}}|[F_{n,\vv{r}},\pi_{n,\vv{r}}(a)] \;.
$$
Let $\xi_{i,j}$ be a generator as in \eqref{eq:conj}. From the proof of Proposition~\ref{prop} we know that
\begin{equation}\label{eq:Fcomm}
[F_{n,\vv{r}},\pi_{n,\vv{r}}(\xi_{i,j})]=
\text{\footnotesize$\bigg(\!\!\begin{array}{cr}0 & \!-1 \\ 1 & 0 \end{array}\!\!\bigg)$}
\Big\{\pi_{n,\vv{r}}^+(\xi_{i,j})-\pi^-_{n,\vv{r}}(\xi_{i,j})\Big\} \;,
\end{equation}
and $\pi_{n,\vv{r}}^+(\xi_{i,j})-\pi^-_{n,\vv{r}}(\xi_{i,j})$ is either zero or a weighted shift with matrix coefficients
bounded by $q^{m_k}$ on $\mc{V}^n_{k-1}\cap\mc{V}^n_k$, for each $1\leq k\leq n$.
Furthermore, for $\vec{m}$ satisfying \eqref{eq:capconstr}, one has $\|\vec{m}\|_1\leq nm_k$.
Since the sequence $\{n\, m_kq^{m_k}\}_{m_k\geq 0}$ is bounded, $|D_{n,\vv{r}}|[F_{n,\vv{r}},\pi_{n,\vv{r}}(a)]$
is bounded on each $\mc{V}^n_{k-1}\cap\mc{V}^n_k$, and then on the whole Hilbert space.

Since $\pi_{n,\vv{r}}^\pm(\xi_{i,j})$ is either zero or a weighted shift on each $\mc{V}^n_{k-1}\cap\mc{V}^n_k$,
there $[|D_{n,\vv{r}}|,\pi^\pm_{n,\vv{r}}(\xi_{i,j})]$ is proportional to $\pi^\pm_{n,\vv{r}}(\xi_{i,j})$, hence
bounded. This establishes the commutator condition.

Next, the multiplicity $\mu_\lambda$ of the eigenvalue $\lambda\in\N$ of $|D_{n,\vv{r}}|$ is given by the number of vectors 
$\vec{m}$ satisfying $\|\vec{m}\|_1= \lambda$.
With the notation $k_i:=m_1+m_2+\ldots+m_i+i$, this $\mu_\lambda$ is the number of $\vec{k}\in\N^n$ satisfying $1\leq k_1<k_2<\ldots<k_n=\lambda+n$,
that is the number of $n-1$ partitions of $\lambda+n-1$. So $\mu_\lambda=\binom{\lambda+n-1}{n-1}$.
Since the latter is a polynomial of order $n-1$ in $\lambda$, $\sum_{\lambda\geq 1}\mu_\lambda\lambda^{-k}<\infty$
for all $n-1-k<-1$, that means $k>n$ as expected.
\end{proof}

The metric dimension of the spectral triple in Proposition~\ref{prop:st} coincides with the classical \emph{complex} dimension: 
$n=\dim_{\C}\WP(\mv{\ell})$. One gets additional spectral triples of any dimension $k<n$ by pulling back spectral triples from spaces $\WP_q(\ell_0,\ldots,\ell_k)$.

One possible generalization of the previous proposition goes as follows. 
Recall first that the Lipschitz norm of a function $g:\R\to\R$ is defined as 
$$
\|g\|_{\mathrm{Lip}}:=\sup_{t\neq s}\left|\frac{g(t)-g(s)}{t-s}\right| \;.
$$
A function is Lipschitz continuous if $\|g\|_{\mathrm{Lip}}<\infty$. Lipschitz continuous functions are
a.e.~
differentiable, and their Lipschitz norm coincides with $\|g'\|_{\infty}$ (the sup norm
of the derivative).

With $\lambda: \R_{\geq 0} \to \R_{\geq 0}$ an increasing function,
replace the operator \eqref{eq:lambdaD} by the more general
$$
|D^\lambda_{n,\vv{r}}|\ket{\vec{m}}=\lambda(\|\vec{m}\|_1)\ket{\vec{m}}\;, \quad \forall\;\vec{m}\in\N^n\;,
$$
and define $D_{n,\vv{r}}^\lambda:=|D_{n,\vv{r}}^\lambda|F_{n,\vv{r}}$. Then:

\begin{prop}\label{prop:stB}
If $\lambda$ is Lipschitz continuous, the datum 
$
\big(\,
\A(\WP_q(\mv{\ell}))
\,,\,
\HH_{n,\vv{r}}
\,,\,
\pi_{n,\vv{r}}
\,,\,
D^\lambda_{n,\vv{r}}
\,,\,
\gamma_{n,\vv{r}}
\,\big)
$
is an even spectral triple.
\end{prop}
\begin{proof}
The proof of Proposition~\ref{prop:st} can be repeated with minor changes. Following that proof, 
one has to show that the operators
$$ 
\textup{i)} \quad [|D^\lambda_{n,\vv{r}}|,\pi^\pm_{n,\vv{r}}(\xi_{i,j})] \qquad \textup{and} \qquad
\textup{ii)} \quad |D^\lambda_{n,\vv{r}}|\big\{\pi_{n,\vv{r}}^+(\xi_{i,j})-\pi^-_{n,\vv{r}}(\xi_{i,j})\big\}
$$
are bounded on $\mc{V}^n_{k-1}\cap\mc{V}^n_k$, for each $1\leq k\leq n$ and for each $i,j$
(cf.~\eqref{eq:Fcomm}).

Concerning i): for a weighted shift like \eqref{eq:WS}, it holds that 
$$
[|D^\lambda_{n,\vv{r}}|,S(\vec{h},c)]\ket{\vec{m}}=c(\vec{m})\big\{\lambda(\|\vec{m}\|_1
+\|\vec{h}\|)-\lambda(\|\vec{m}\|_1)\big\}\ket{\vec{m}+\vec{h}} \, .
$$
Due to the Lipshitz condition, the matrix coefficients are bounded by $\|\lambda\|_{\mathrm{Lip}}\|h\|_1$
times the operator norm of $S(\vec{h},c)$. Hence the commutator is bounded. Since each $\pi^\pm_{n,\vv{r}}(\xi_{i,j})$
is a weighted shift restricted to $\mc{V}^n_{k-1}\cap\mc{V}^n_k$, this proves that the commutators i) are bounded.

Concerning ii): from Lipschitz continuity we deduce $|\lambda(t)|\leq t$, which means that
\begin{equation}\label{eq:boundexp}
\widetilde{\lambda}(t) = q^{t/n}\lambda(t)
\end{equation}
is a bounded function. 
Now, the operator ii) is a weighted shift with matrix coefficients bounded by $\lambda(\|\vec{m}\|_1)q^{m_k}$.
Since $\lambda$ is increasing, and $\|\vec{m}\|_1\leq nm_k$ for every $\vec{m}$ satisfying \eqref{eq:capconstr},
the matrix coefficients of ii) are bounded by $q^{m_k} \lambda(nm_k)$. Calling $t:=nm_k$,
as said the function $\widetilde{\lambda}(t)=q^{t/n} \lambda(t)$ is bounded, thus the operator ii) is bounded.
\end{proof}

From previous proposition one can get spectral triples of arbitrary metric dimension $d\geq n$.

\begin{prop}
Let $d\geq n$ be a real number. With the choice $\lambda(t):=t^{n/d}$
the spectral triple in Proposition~\ref{prop:stB} has metric dimension $d$.
\end{prop}

\begin{proof}
Eigenvalues of $|D^\lambda_{n,\vv{r}}|$ are given by $\lambda(j)$, with $j\in\N$. The multiplicity of
$\lambda(j)$ is $\binom{j+n-1}{n-1}$, which is a polynomial of order $n-1$ in $j$ (as in the proof of
Proposition~\ref{prop:st}). Let us write the leading term in the zeta-function of the Dirac operator:
$$
\tr(|D^\lambda_{n,\vv{r}}|^{-s})=\sum_{j\geq 1}j^{n-1}\lambda(j)^{-s}
+\text{lower order terms},
$$
where the trace is on the orthogonal complement of the kernel of $D^\lambda_{n,\vv{r}}$.
For $\lambda(t)=t^{n/d}$, this is convergent for $\Re(s)\geq d$ and has a pole at $s=d$,
proving that the metric dimension is $d$.
\end{proof}

\begin{rem}
For $d<n$,  
with $\lambda(t):=t^{n/d}$ the function $\widetilde{\lambda}(t) = q^{t/n}\lambda(t)$ is still bounded.
What fails is Lipschits continuity: the derivative $\lambda'(t) = \frac{n}{d}\,t^{n/d-1}$ is unbounded if $n/d-1>0$.
The example of $\CP^1_q$, that is the standard Podle{\'s} \cite{DDLW07}, would suggest that while boundedness of the function 
$\widetilde{\lambda}(t)$ in \eqref{eq:boundexp} is necessary in order to have a spectral triple, the Lipschitz condition is sufficient but not necessary.
Were this to be true, the above construction would yield spectral triples of any metric dimension,
even $0^+$ with $\lambda(t):=q^{-\epsilon t}$ (for any $0<\epsilon<1/n$). 
\end{rem}

\begin{rem}
The spectral triples above have no classical analogue (the representation become trivial for $q=1$).
For the quantum projective space $\CP^n_q$ there are additional equivariant, \mbox{$0^+$-summable},
spectral triples, which for $q=1$ give the Dolbeault-Dirac
operator of $\CP^n$ twisted with a line bundle \cite{DD09} (see also \cite{DDL08}).
For quantum weighted projective spaces it is not clear how to get a $q$-analogue of the Dolbeault-Dirac
operator (a crucial ingredient in the construction --- the action of $\mc{U}_q(\mathfrak{su}(n+1))$ --- is missing in these cases).
\end{rem}

\section{Principal bundle structures}\label{sec:sei}

It is well-known that the inclusion $\A(\CP^1_q) \hookrightarrow \A(S^3_q)$ is a quantum principal bundle \cite{BM96}.
On~the other hand, if $\ell_0\neq 1$, $\A(\WP_q(\ell_0,\ell_1)) \hookrightarrow \A(S^3_q)$
is not a quantum principal bundle (nor is a more general principal comodule algebra), since surjectivity of the 
canonical map fails \cite{BF12}. For $p=\ell_0\ell_1$, the inclusion 
$\A(\WP_q(\ell_0,\ell_1)) \hookrightarrow \A(L_q(p;\ell_0,\ell_1))$ is a quantum principal bundle:
this was proved in \cite{BF12} for $\ell_0=1$ and in \cite{AKL14} for general weights $\ell_0,\ell_1$.
In \cite{ABL14} there is the case of quantum lens spaces in any dimension $n$  
but with weights all equal to 1 and any integer $p$; so that the `base space' is now a quantum projective space.

In this section, we are going to extend these results to our class of quantum lens and weighed projective spaces, showing that the inclusion
$\A(\WP_q(\mv{\ell})) \hookrightarrow \A(L_q(p,\mv{\ell}))$, for $\mv{\ell}=\mv{p}^{\,\sharp}$ with $\mv{p}$ pairwise coprime, 
is a quantum principal $\U(1)$-bundle.

Not needing the full fledged theory, we content ourself with the following definition \cite{BM96,Ha96}.
Let $H=\A(\U(1))$ be the Hopf $*$-algebra generated by a unitary group-like element $u$. 
Let $A$ be a right comodule algebra over 
$H=\A(\U(1))$, that is there is a coaction,
$$
\delta : A \to A \otimes H \, ,
$$
with $B:=A^{\text{co}H}$ the subalgebra of $A$ consisting of coinvariant elements.
One says that $A$ is \emph{principal} or that $B \hookrightarrow A$ is a \emph{quantum principal $\U(1)$-bundle}, if the canonical
map, 
$$
\mathrm{can} : A \otimes_{B} A \to A \otimes H \, , \quad x \otimes y \mapsto x \, \delta (y) \, ,
$$
is an isomorphism. Indeed, being $H$ cosemisimple with bijective antipode, the surjectivity of the canonical map implies its bijectivity  and also faithfully flatness of the extension $B \hookrightarrow A$. \linebreak
With $H=\A(\U(1))$, the algebra $A$ gets naturally graded, $A=\bigoplus_{k\in\Z}\mc{L}_k$ where
$$
\mc{L}_k:=\big\{ a\in A: \delta(a)=a\otimes u^{-k} \big\} \;,
$$
and the principality of the algebra $A$ becomes then equivalent to $A$ being \emph{strongly} $\Z$-graded \cite[Cor.~I.3.3]{NVO82}, 
that is $\mc{L}_k \mc{L}_{k'} = \mc{L}_{k+k'}$.
An efficient way to establish this is by use of the so-called strong connection,
$A$ being principal (or equivalently strongly $\Z$-graded) if and only if such a strong connection exists \cite{Ha96}.
For the case at hand with $H=\A(\U(1))$, a \emph{strong connection} is a linear map $\omega:H\to A\otimes A$ satisfying the following conditions:
\begin{align}
 &\omega(1) = 1\otimes 1\;, \notag \\
 &\omega(u^k)\in\mc{L}_{-k}\otimes \mc{L}_{k} &\hspace{-2cm} \forall\;k\in\Z \,, \notag \\
 \qquad\qquad
 \sum\nolimits_i&\omega(u^k )_i^{[1]}\omega(u^k)_i^{[2]} =1 &\hspace{-2cm} \forall\;k\in\Z \, . \label{eq:22nontrivial}
\end{align}
Here we used the notation
$$
\omega(h) = \sum\nolimits_i\omega(h)_i^{[1]}\otimes\omega(h)_i^{[2]} \;,\qquad \textup{for} \,\, h \in H \, .
$$
As a consequence, the matrix $E_k$ with entries
\begin{equation}\label{eq:proj}
(E_k)_{ij}:=\omega(u^k)_i^{[2]}\omega(u^k)_j^{[1]}
\end{equation}
is a coinvariant idempotent, that is its entries are in the algebra $B = \mc{L}_0$ of coinvariants.
Thus, the principality of $A$  implies that each $\mc{L}_{k}$ is finitely generated and projective as left and right $\mc{L}_0$-module. In fact, one can easily show the left (respectively right) $\mc{L}_0$-module isomorphisms $B^{N_k}E_k\simeq\mc{L}_{k}$ 
(and $E_kB^{N_k}\simeq\mc{L}_{-k}$), where $N_k$ is the size of $E_k$.

Back to quantum lens and weighted projective spaces. Firstly, dually to the $\U(1)$-action \eqref{eq:act}, one has
a coaction of the Hopf algebra $\A(\U(1))$ on the sphere:
$$
\A(S^{2n+1}_q)\to \A(S^{2n+1}_q)\otimes \A(\U(1)) \;,\quad
z_i\mapsto z_i\otimes u^{\ell_i}\; , \quad \forall\;i=0,\ldots,n.
$$
Next, we take $A=\A(L_q(p,\mv{\ell}))$ as in Theorem~\ref{thm:13}, that is for $\mv{\ell}=\mv{p}^{\,\sharp}$ with $\mv{p}$ pairwise coprime and 
$p=p_0p_1\ldots p_n$. On the generators of $A$ the previous coaction becomes:
\begin{equation}\label{eq:coactA}
\zeta_i \mapsto \zeta_i\otimes u^p \, , \qquad x_i \mapsto x_i\otimes 1 \, , \qquad \forall\;i=0,\ldots,n \, .  
\end{equation}
The subalgebra of coinvariant elements is clearly $B = \A(\WP_q(\mv{\ell}))$.

However, the coaction \eqref{eq:coactA} is not quite the one we are after.
Classically the lens space $L(p,\mv{\ell})$ is a principal bundle over $\WP(\mv{\ell})$  with structure group $\U(1)/\Z_p\simeq\U(1)$.
In algebraic terms, this amounts to taking as structure Hopf algebra $H$ the Hopf $*$-subalgebra of $\A(\U(1))$ generated by $u':=u^p$, which clearly is still isomorphic to $\A(\U(1))$. Renaming $u'$ to $u$, the `correct' coaction $\delta:A\to A\otimes H$ on generators becomes:
\begin{equation}\label{eq:coactB}
\delta(\zeta_i)=\zeta_i\otimes u \, , \qquad \delta(x_i)=x_i\otimes 1 \, , \qquad \forall\;i=0,\ldots,n \, ,  
\end{equation}
for which the subalgebra of coinvariant elements is again $B = \A(\WP_q(\mv{\ell}))$.

We next show that the algebra inclusion $\A(\WP_q(\mv{\ell})) \hookrightarrow \A(L_q(p,\mv{\ell}))$ 
(for the coaction $\delta$ in \eqref{eq:coactB}) is a quantum principal $\U(1)$-bundle. 
We do this by establishing in general, the existence of  a strong connection and by providing recursive relations
that in principle would allow one to write down explicitly the connection case by case.

\begin{prop}\label{prop:aibi}
Consider the commuting generators $x_i$, $i=1, \dots n$, of $\A(\WP_q(\mv{\ell}))$. Then \\
i) There exists $a_0,a_1,\ldots,a_n\in\C[x_1,\ldots,x_n]$ such that
\begin{subequations}
\begin{equation}\label{eq:solu}
a_0\zeta_0\zeta_0^*+a_1\zeta_1\zeta_1^*+\ldots+a_n\zeta_n\zeta_n^*=1 \;.
\end{equation}
ii) There exists $b_0,b_1,\ldots,b_n\in\C[x_1,\ldots,x_n]$ such that
\begin{equation}\label{eq:solu-b}
b_0\zeta_0^*\zeta_0+b_1\zeta_1^*\zeta_1+\ldots+b_n\zeta_n^*\zeta_n=1 \;.
\end{equation}
\end{subequations}
\end{prop}

\noindent
Note that due to \eqref{eq:relLensH} and \eqref{eq:relLensL}, the products $\zeta_i\zeta_i^*$ and $\zeta_i^*\zeta_i$ all
belong to the commutative subalgebra generated by $x_0,\ldots,x_n$. Using \eqref{eq:relLensG} we can eliminate $x_0$, hence
it makes sense to look for solutions of the above equations that are polynomials in $\C[x_1,\ldots,x_n]$.

For $n=1$, we give a proof using B{\'e}zout's identity as in Lemma~\ref{lemma:Bez} for the principal ideal domain $R:=\C[x_1]$.
For general $n$, we cannot use B{\'e}zout's identity since $\C[x_1,\ldots,x_n]$ is not a principal ideal domain
if $n\geq 2$, but we can use Hilbert's weak Nullstellensatz, which states that the only ideal representing the empty
variety is the entire polynomial ring.

\begin{lemma}[Hilbert's weak Nullstellensatz]\label{lemma:Null}
An ideal $I\subset\C[x_1,\ldots,x_n]$ contains $1$ if and only if the polynomials in $I$ do not have any common zero, 
i.e. $\mathcal{Z}(I)=\emptyset$.
\end{lemma}

For arbitrary $n$, we give two alternative proofs of Proposition~\ref{prop:aibi}, one using the Nullstellensatz and a second one
which is more explicit and which will be useful later on to compute some pairings between K-theory and K-homology.
We will only prove point (i) of the proposition, the proof of point (ii) being clearly analogous.

\begin{proof}[Proof 0: B{\'e}zout's identity ($n=1$)]
Let $R=\C[x_1]$.
Since 
$a:=\zeta_1\zeta_1^*=
x_1^{p_1}$ and
$b:=\zeta_0\zeta_0^* =\prod_{k=0}^{p_0-1}(1-q^{-2k}x_1)$ are coprime (they have
no common zeros), by Lemma~\ref{lemma:Bez} there exist $x,y\in R$ such that
$ax+by=1$, which is \eqref{eq:solu} except for a different notation.
\end{proof}

\begin{proof}[Proof 1: Hilbert's Nullstellensatz]
For $0\leq k\leq n$,
let $I_k\subset\C[x_1,\ldots,x_n]$ be the ideal generated by $\{\zeta_j\zeta_j^*\}_{j=k}^n$.
Clearly, the zero loci satisfy $\mathcal{Z}(I_k)=\mathcal{Z}(I_{k+1})\cap\mathcal{Z}(\{\zeta_k\zeta_k^*\})$.

By induction on $k\geq 1$ one proves that $(x_1,\ldots,x_n)\in\mathcal{Z}(I_k)$
if and only if $x_j=0$ for all $j\geq k$. This is is true for $k=n$, since $\zeta_n\zeta_n^*=x_n^{p_n}$.
If it is true for some $k$, by simplifying \eqref{eq:relLensH} using $x_k=x_{k+1}=\ldots=x_n=0$ one gets
$\zeta_{k-1}\zeta_{k-1}^*=x_{k-1}^{p_{k-1}}$, which vanishes only if $x_{k-1}=0$. This proves
the inductive step.

Now $\mathcal{Z}(I_1)=\{0\}$, but $x_1=\ldots=x_n=0$ implies $x_0=1$ and then $\zeta_0\zeta_0^*=1\neq 0$.
Thus $\mathcal{Z}(I_0)=\emptyset$ and from Lemma~\ref{lemma:Null} it follows that $1\in I_0$.
But any element in $I_0$ is a linear combination $a_0\zeta_0\zeta_0^*+a_1\zeta_1\zeta_1^*+\ldots+a_n\zeta_n\zeta_n^*$,
with coefficients $a_0,\ldots,a_n\in \C[x_1,\ldots,x_n]$, thus proving that \eqref{eq:solu} must admit a solution.
\end{proof}

\begin{proof}[Proof 2: by induction]
For $0\leq k\leq n+1$ consider the following statement: there exist elements $a_{0,k},\ldots,a_{n,k}\in\C[x_1,\ldots,x_n]$ such that
\begin{gather}\label{eq:Pik}
\sum_{i=0}^{k-1}a_{i,k}\zeta_i\zeta_i^*+
\sum_{i=k}^na_{i,k}z_iz_i^*=1 \;. \tag{eq$_k$}
\end{gather}
We prove this by induction on $k$. It is understood that an empty sum is zero.
The above is true for $k=0$ with $a_{0,0}=a_{1,0}=\ldots=a_{n,0}=1$.

Next, one takes the $p_k$-th power of both sides of \eqref{eq:Pik}. Note that the set of monomials $\zeta_i\zeta_i^*$, $z_iz_i^*$ 
and $a_{i,k}$ are mutually commuting. From the multinomial formula, it follows that  
$$
\sum_{s_0+\ldots+s_n=p_k}[s_0,s_1,\ldots,s_n]!
\left(\prod_{j=0}^na_{j,k}^{s_j}\right)
\left(\prod_{j=0}^{k-1}(\zeta_j\zeta_j^*)^{s_j}\right)
\left(\prod_{j=k}^n(z_jz_j^*)^{s_j}\right)= 1 \;,
$$
where
$$
[s_0,s_1,\ldots,s_n]!=\frac{(s_0+s_1+\ldots+s_n)!}{s_0!\,s_1!\,\ldots\hspace{1pt}s_n!}
$$
is the ($q=1$) multinomial coefficient. We break the sum as follows:
$$
\sum_{s_0+\ldots+s_n=p_k}=
\sum_{\substack{ s_k=p_k \\[1pt] s_i=0\;\forall\;i\neq k} }
+\sum_{i\neq k}\sum_{\substack{s_0+\ldots+s_n=p_k \\[1pt] s_0=s_1=\ldots=s_{i-1}=0 \\[1pt] s_i\neq 0}} \, .
$$
Using \eqref{eq:relLensH}:
$$
\zeta_k\zeta_k^* =
x_k^{p_k}+A_k(x_k,\ldots,x_n)\sum\nolimits_{j>k}x_j \,  ,
$$
where $A_k(x_k,\ldots,x_n)$ is a polynomial of $x_k,\ldots,x_n$. Hence:
$$
(z_kz_k^*)^{p_k}=
\zeta_k\zeta_k^*-A_k(x_k,\ldots,x_n)\sum\nolimits_{j>k}z_jz_j^* \;.
$$
Then (eq$_{k+1}$) is satisfied by defining recursively:
\begin{list}{$\bullet$}{\leftmargin=1em \itemsep=0pt}
\item for $i<k$:
$$
a_{i,k+1}:=
\sum_{\substack{s_i+\ldots+s_n=p_k \\[1pt] s_i\geq 1}}
[s_i,s_{i+1},\ldots,s_n]!
\left(\prod_{j=i}^na_{j,k}^{s_j}\right)
(\zeta_i\zeta_i^*)^{s_i-1}
\left(\prod_{j=i+1}^{k-1}(\zeta_j\zeta_j^*)^{s_j}\right)
\left(\prod_{j=k}^n(z_jz_j^*)^{s_j}\right) \;,
$$
\item for $i=k$:
$$
a_{k,k+1}:=a_{k,k}^{p_k} \;,
$$
\item and for $i>k$:
\begin{multline*}
a_{i,k+1}:=
\sum_{\substack{s_i+\ldots+s_n=p_k \\[1pt] s_i\geq 1}}
[s_i,s_{i+1},\ldots,s_n]!
\left(\prod_{j=i}^na_{j,k}^{s_j}\right)
(z_iz_i^*)^{s_i-1}
\left(\prod_{j=i+1}^n(z_jz_j^*)^{s_j}\right) \\
-a_{k,k}^{p_k}A_k(x_k,\ldots,x_n) \; .
\end{multline*}
\end{list}
The proof of Proposition~\ref{prop:aibi} is completed if one puts $a_i:=a_{i,n+1}$ for all $i=0,\ldots,n$.
\end{proof}

\begin{thm}\label{eq:recurs}
For any weight vector $\mv{\ell}=\mv{p}^{\,\sharp}$, with $\mv{p}$ pairwise coprime,
a strong connection on $\A(L_q(p,\mv{\ell}))$, with $p:=p_0p_1\ldots p_n$, is defined recursively by
\begin{align*}
\omega(1) &=1\otimes 1 \;,\\
\omega(u^k) &=\sum_{i=0}^na_i\zeta_i \, \omega(u^{k-1})\, \zeta^*_i
 \;, \qquad\textup{for}\;\; k\geq 1, \\
\omega(u^k) &=\sum_{i=0}^nb_i\zeta_i^* \, \omega(u^{k+1})\, \zeta_i
 \;, \qquad\textup{for}\;\; k\leq -1,
\end{align*}
where $a_i$,$b_i$ are the polynomials in Proposition~\ref{prop:aibi}.
Moreover, the quantum principal $\U(1)$-bundle $\A(L_q(p,\mv{\ell}))$ over $\A(\WP_q(\mv{\ell}))$ is not trivial.
\end{thm}
\begin{proof}
The only non-trivial condition to check is the last one in \eqref{eq:22nontrivial}, which we show by induction. 
If $k\geq 1$:
$$
\omega(u^k)=\sum_{i=0}^na_i\sum\nolimits_j\zeta_i \, \omega(u^{k-1})_j^{[1]}\otimes\omega(u^{k-1})_j^{[2]}\, \zeta_i^* \;.
$$
By the inductive hypothesis:
$$
\sum\nolimits_i\omega(u^k)_i^{[1]}\omega(u^k)_i^{[2]} =\sum_{i=0}^na_i\zeta_i\zeta_i^*=1 \, , 
$$
having used \eqref{eq:solu}. Similarly, by using induction and \eqref{eq:solu-b} one shows \eqref{eq:22nontrivial} for $k\leq -1$.

As for the non-triviality of the bundle, one can repeat verbatim the proof of \cite[Lem.~3.4]{BF12},
which only uses the fact that the unique invertible elements of $\A(S^{2n+1}_q)$ are the multiples of $1$.
Alternatively, the statement is a consequence of Proposition~\ref{prop:21} below.
\end{proof}

For $n=1$ one can compute explicitly the two polynomials $a_0,a_1$ (cf.~also \cite[Prop.~6.4]{AKL14}).

We need some notations. We define the $q$-analogue of an integer $k$, for $q\neq 1$, as
$$
[k]_q:=\frac{q^k-q^{-k}}{q-q^{-1}} \;.
$$
Clearly $[k]_q\to k$ when $q\to 1^-$.
The $q$-factorial is defined recursively by
$$
[0]_q!:=1 \;,\qquad
[k]_q!:=[k]_q\cdot [k-1]_q! \quad\text{for}\;\;k\geq 1\,. 
$$
For $0\leq k\leq m$, we define the $q$-binomial through the identity of polynomials in $t$:
\begin{equation}\label{eq:qbinom}
\prod_{l=0}^{m-1}\big(1+q^{2l}t\big)=
\sum_{k=0}^{m}\sqbn{m}{k}q^{k(m-1)}t^k \; , 
\end{equation}
with the convention that an empty sum is $0$ and an empty product is $1$. Through the
substitution $t\to q^{-2(m-1)}t$ one verifies that the $q$-binomial is invariant under $q\to q^{-1}$.
From the recursive formula:
$$
\sqbn{m+1}{k}=q^{-k}\sqbn{m}{k}+q^{m-k+1}\sqbn{m}{k-1}
$$
one deduces by induction that
$$
\sqbn{m}{k}=\frac{[m]_q!}{[k]_q![m-k]_q!} \;.
$$

\begin{prop}\label{prop:coeffn1}
For $n=1$, two elements $a_0,a_1\in\C[x_1]$ satisfying \eqref{eq:solu} are given by:
$$
a_0(x_1)=\sum_{k=1}^{p_1}\binom{p_1}{k}f(x_1)^{k-1} \big\{1-f(x_1)\big\}^{p_1-k} \;,\qquad
a_1(x_1)=\left(\frac{1-f(x_1)}{x_1}\right)^{p_1} \;,
$$
where
\begin{equation}\label{eq:at}
f(t):=\sum_{k=0}^{p_0}\sqbn{p_0}{k}q^{-k(p_0-1)}(-t)^k \;.
\end{equation}
\end{prop}

\begin{proof}
Note that ${(1-f(t))} / {t}$ is a well defined polynomial in $t$.
From \eqref{eq:relLensH} and \eqref{eq:qbinom} one has:
$$
\zeta_0\zeta_0^*=\prod\nolimits_{k=0}^{p_0-1}(1-q^{-2k} x_1)=f(x_1) \;,
$$
being $x_1=1-x_0$, with  $f(t)$ given by \eqref{eq:at}.
From
$x_1=z_1z_1^*$ we get the algebraic identity
$$
\zeta_0\zeta_0^*+\frac{1-f(x_1)}{x_1} \, z_1z_1^*=1 \;.
$$
We now take the $p_1$-th power and use the binomial formula to get:
$$
\zeta_0\zeta_0^* \, \sum_{k=1}^{p_1}\binom{p_1}{k}(\zeta_0\zeta_0^*)^{k-1} \big\{1-f(x_1)\big\}^{p_1-k}
+\left(\frac{1-f(x_1)}{x_1}\right)^{p_1}\zeta_1\zeta_1^*=1 \;,
$$
where we used $(x_1)^{p_1}=\zeta_1\zeta_1^*$ (cf.~\eqref{eq:relLensH}).
If we call $a_0$ the first sum and $a_1$ the coefficient of $\zeta_1\zeta_1^*$,
the proof is concluded.
\end{proof}

\begin{prop}\label{prop:21}
Let $E=E_{1}$ be the idempotent defined in \eqref{eq:proj}, for the strong connection of Theorem~\ref{eq:recurs};
let $\mc{F}_{1,\vv{r}}$ the $1$-summable Fredholm module of \S\ref{sec:Khom}. Then $\inner{\mc{F}_{1,\vv{r}}, [E]}=-1$.
\end{prop}

\begin{proof}
In the present case $\vv{r}=(r_0)$ is an integer, $0\leq r_0<p_0$, and the Hilbert space is $\ell^2(\N)$ with 
orthonormal basis $\{\ket{m_1},m_1\in\N\}$.
The representations $\pi^\pm_{1,\vv{r}}$ satisfy $\pi^\pm_{1,\vv{r}}(x_i)=0$, for all $i\geq 2$, 
$\pi^\pm_{1,\vv{r}}(x_0)=1-\pi^\pm_{1,\vv{r}}(x_1)$, and
$$
\pi^+_{1,\vv{r}}(x_1)=0\;,\qquad
\pi^-_{1,\vv{r}}(x_1)\ket{m_1}=q^{2r_0+2p_0m_1}\ket{m_1} \;.
$$
Let us write $a\sim b$ if $a-b$ is in the kernel
of both $\pi^+_{1,\vv{r}}$ and $\pi^-_{1,\vv{r}}$. So, $x_i\sim 0$ for all $i\geq 2$ and 
$x_1\sim 1-x_0$. \emph{De facto}, the computation reduces to the case $n=1$.

The pairing with any idempotent $E=(E_{ij})$ is
$$
\inner{\mc{F}_{1,\vv{r}}, [E]}=\tr_{\ell^2(\N)}
\big(\pi^+_{1,\vv{r}}-\pi^-_{1,\vv{r}}\big)\big(\tr(E)\big) \;.
$$
If $E=E_1$ is the idempotent in \eqref{eq:proj}, with strong connection $\omega(u)=\sum_{i=0}^na_i \, \zeta_i\otimes\zeta^*_i$
as in Theorem~\ref{eq:recurs}, we get $\tr(E)=\sum_{i=0}^n\zeta^*_ia_i\zeta_i$.
Modulo elements in the kernel of  $\pi^\pm_{1,\vv{r}}$ the coefficients $a_0(x_1),a_1(x_1)$ are those in Proposition~\ref{prop:coeffn1},
while all other coefficients $a_i$'s are zero.
From $x_1\,\zeta_0=q^{2p_0}\zeta_0\,x_1$ and $x_1\,\zeta_1\sim\zeta_1\,x_1$ we get
\begin{align*}
\tr(E)
&\sim a_0(q^{2p_0}x_1)\zeta^*_0\zeta_0+a_1(x_1)\zeta^*_1\zeta_1 \\
&\sim a_0(q^{2p_0}x_1)\prod_{k=1}^{p_0}(1-q^{2k}x_1)
+a_1(x_1) (x_1)^{p_1} \\
&\sim a_0(q^{2p_0}x_1)\prod_{k=1}^{p_0}(1-q^{2k}x_1)
+\big\{1-f(x_1)\big\}^{p_1}
\end{align*}
where we used \eqref{eq:relLensL} and then Proposition~\ref{prop:coeffn1};
with $f(t)$ the function in \eqref{eq:at}.
In the above expressions, we denoted by $a_0(q^{2p_0}x_1)$ the polynomial obtained from the element
$a_0(x_1)$ in \eqref{prop:coeffn1} with a replacement $x_1\mapsto q^{2p_0}x_1$.
From \eqref{eq:qbinom}:
$$
\prod_{k=1}^{p_0}(1-q^{2k}x_1)=\prod_{l=0}^{p_0-1}(1-q^{2l}q^2x_1)=
\sum_{k=0}^{p_0}\sqbn{p_0}{k}q^{k(p_0-1)}(-q^2x_1)^k=
f(q^{2p_0}x_1)
 \;.
$$
But the identity $a_0\zeta_0\zeta_0^*+a_1\zeta_1\zeta_1^*=1$ reads 
$a_0(q^{2p_0}x_1)f(q^{2p_0}x_1)+a_1(q^{2p_0}x_1)\cdot(q^{2p_0}x_1)^{p_1}=1$, when $x_1$ is replaced by $q^{2p_0}x_1$. 
Therefore:
$$
\tr(E)
\sim 1-\big\{1-f(q^{2p_0}x_1)\big\}^{p_1}+\big\{1-f(x_1)\big\}^{p_1} \;.
$$
Since $x_1$ is diagonal in both representations and $f(0)=1$,
$$
\inner{\mc{F}_{1,\vv{r}}, [E]}
=\sum_{m_1\geq 0}\Big[\big\{1-f(q^{2p_0}x_1)\big\}^{p_1}-\big\{1-f(x_1)\big\}^{p_1}\Big]_{x_1=q^{2r_0+2p_0m_1}} \;.
$$
From the condition $0<q<1$ it follows that 
$$
|1-f(x_1)|_{x_1=q^{2r_0+2p_0m_1}}\leq q^{2p_0m_1}\sum_{k=1}^{p_0}\sqbn{p_0}{k}q^{-k(p_0-1)}
 \;,
$$
and similar for $1-f(q^{2p_0}x_1)$.
Thus, the series in $m_1$ is bounded by $q^{2p_0p_1m_1}$ times a constant. 
By Weierstrass M-test it is absolutely convergent in the open interval $[0,1[$ and then continuous. 
Being integer-valued in $]0,1[$, it is constant in the interval (including $0$) and it can be computed for $q\to 0^+$, by inverting summation and this limit. Now,
$$
1-f(q^{2p_0}x_1)\big|_{x_1=q^{2r_0+2p_0m_1}}
= 1-\prod_{k=0}^{p_0-1}(1-q^{2(p_0m_1+r_0+p_0-k)})
= 1-1+O(q)=O(q)
\;,
$$
since $p_0m_1+r_0+p_0-k\geq p_0-k\geq 1$.
Moreover
$$
1-f(x_1)\big|_{x_1=q^{2r_0+2p_0m_1}}=
1-\prod_{k=0}^{p_0-1}(1-q^{2(p_0m_1+r_0-k)})=:c_{m_1}
\;.
$$
Being $0\leq r_0<p_0$, for $m_1=0$ at least one term in the product
has $k=r_0$, so the product is zero and $c_{m_1=0}=1$. On the other hand,
if $m_1\geq 1$, then $p_0m_1+r_0-k\geq p_0-k\geq 1$ and so the product
is $1+O(q)$ and $c_{m_1}=O(q)$. We conclude that
$$
\lim_{q\to 0^+}\inner{\mc{F}_{1,\vv{r}}, [E]}
=-\sum_{m_1\geq 0}\lim_{q\to 0^+}c_{m_1}^{p_1}=-c_0^{p_1}=-1 \;. \vspace{-10pt}
$$
\end{proof}

Proposition~\ref{prop:21} allows us to computes the pairing of $P$ with Fredholm modules that are pullbacks from $\A(\WP_q(\ell_0,\ell_1))$, 
thanks to the fact that for $n=1$ one has an explicit expression for the coefficients in \eqref{eq:solu} (and then for the trace of $P$).

In order to compute the pairing with all the other Fredholm modules, one would need the expression of $P$
for arbitrary $n$. Such a computation seems intractable in full generality. We study some interesting examples in the next section.

\section{Examples}\label{sec:sette}
Let us compute the coefficients in \eqref{eq:solu} for particular values of the
weight vector. In all of the examples below, we will assume that all the weights are
equal but one, say $\ell_{i_0}$. In the cases $i_0=0$ and $i_0=n$, the coefficients in
\eqref{eq:solu} can be explicitly computed.  
We start with these two examples, in reverse order.

\begin{prop}
Let $\mv{\ell}=\mv{p}^{\,\sharp}$ with $p_i=1$ for all $i\neq n$.
A set of elements $a_0,\ldots,a_n$ satisfying \eqref{eq:solu} is given by:
$$
a_n=1 \;,\qquad
a_0=a_1=\ldots=a_{n-1}=\sum_{k=0}^{p_n-1}(x_n)^k \;.
$$
\end{prop}
\begin{proof}
Since $\zeta_i\zeta_i^*=z_iz_i^*=x_i$ for all $i\neq n$ and $\zeta_n\zeta_n^*=(z_nz_n^*)^{p_n}=x_n^{p_n}$
(using the relation \eqref{eq:relLensH} and the fact that $z_n$ is normal), we get
\begin{align*}
a_0\zeta_0\zeta_0^*+\ldots+a_n\zeta_n\zeta_n^* & = (x_0+x_1+\ldots+x_{n-1})\sum_{k=0}^{p_n-1}(x_n)^k+x_n^{p_n} \\
& = (1-x_n) \sum_{k=0}^{p_n-1}(x_n)^k+x_n^{p_n} = (1-x_n^{p_n}) + x_n^{p_n} \\
& = 1 \; .
\end{align*}
Hence \eqref{eq:solu} is satisfied.
\end{proof}

\begin{prop}
Let $\mv{\ell}=\mv{p}^{\,\sharp}$ with $p_i=1$ for all $i\neq 0$.
With $f(t)$ the function in \eqref{eq:at}, a set of elements $a_0,\ldots,a_n$ satisfying \eqref{eq:solu} is given by:
$$
a_0=1 \;,\qquad
a_1=\ldots=a_n=\frac{1-f(x_1+\ldots+x_n)}{x_1+\ldots+x_n} \;.
$$
\end{prop}
\begin{proof}
We know from the proof of Proposition~\ref{prop:coeffn1} that $ (1-f(t) ) / t $ is a well
defined polynomial of $t$. Now $t:=1-x_0 = x_1+x_2+\ldots+x_n$, and it still holds 
that $\zeta_0\zeta_0^*=f(t)$. Condition \eqref{eq:solu} reduces to
$$
\zeta_0\zeta_0^*+\frac{1-f(t)}{t}(z_1z_1^*+\ldots+z_nz_n^*)=1 \; , 
$$
since $z_1z_1^*+\ldots+z_nz_n^* = x_1+\ldots+x_n = t $ and $\zeta_0\zeta_0^*+1-f(t)=1$. 
\end{proof}

\begin{rem}
Fix $i_0\in\{0,\ldots,n\}$ and let $\mv{\ell}=\mv{p}^{\,\sharp}$ with $p_i=1$ for all $i\neq i_0$. \\
From previous examples one may think that a solution to \eqref{eq:solu} is given by
$a_{i_0}=1$ and $a_i=(1-\zeta_{i_0}\zeta_{i_0}^*)/(1-x_{i_0})$. Indeed, with these choices
$$
\sum a_i\zeta_i\zeta_i^*=\zeta_{i_0}\zeta_{i_0}^*
+\frac{1-\zeta_{i_0}\zeta_{i_0}^*}{1-x_{i_0}}\sum_{i\neq i_0}z_iz_i^*=1
$$
is a simple algebraic identity, since $\sum_{i\neq i_0}z_iz_i^*=1-x_{i_0}$
simplifies the denominator. Unfortunately, in general
$1-x_{i_0}$ does not divide $1-\zeta_{i_0}\zeta_{i_0}^*$, so that the quotient
is not a polynomial in $x_1,\ldots,x_n$.
For example, for $n=2$, if $\mv{p}=(1,2,1)$ (so $i_0=1$) one has 
$1-\zeta_1\zeta_1^*=1-x_1(x_1+x_2-q^{-2}x_2)$, which is not divisible by $1-x_1$:
that is it does not vanish if $x_1=1$, and arbitrary $x_2$, unless $q=1$.
\end{rem}

\section{On C*-algebras and K-homology}\label{sec:nove}
For $\lambda\in\U(1)$, let $\psi^{(2n+1)}_\lambda$ be the representation of $\A(S^{2n+1}_q)$
given in \S\ref{sec:irreps} composed with the automorphism $z_n\mapsto\lambda z_n$.
Every bounded irreducible representation of $\A(S^{2n+1}_q)$ is isomorphic to a representation
$\psi^{(2k+1)}_\lambda$, for $0\leq k\leq n$, pulled back from $\A(S^{2k+1}_q)$, see \cite{HL04}.

The direct sum of all these representations is faithful --- it is the so-called \emph{reduced atomic representation} ---,
and the $C^*$-completion of $\A(S^{2k+1}_q)$ in the associated norm is the universal $C^*$-algebra $C(S^{2k+1}_q)$
(cf.~\cite[Prop.~10.3.10]{KR83}).

In this section, we study the $C^*$-subalgebra $C(\WP_q(\mv{\ell}))$ completion of $\A(\WP_q(\mv{\ell}))$ in $C(S^{2k+1}_q)$.
It is not obvious whether or not by restriction of $\psi^{(2n+1)}_\lambda$ one gets all equivalence classes of irreducible representations of $C(\WP_q(\mv{\ell}))$, so that their direct sum would be the reduced atomic representation and the
$C^*$-algebra be universal. In fact, classifying irreducible representations of $C(\WP_q(\mv{\ell}))$ goes beyond the scope of this paper.

We limit ourself to exhibit a family of projections in $C(\WP_q(\mv{\ell}))$ and to
compute their pairing with the Fredholm modules $\mc{F}_{k,\vv{r}}$ of \S\ref{sec:Khom}. 
This will show that the classes in K-homology of these Fredholm modules are linearly independent over $\Z$.

It is convenient to use the operators $X_i:=\sum_{j\geq i}x_j$, for $1\leq i\leq n$ (recall that $x_i=z_iz_i^*$),
since one easily verifies that, for any $1\leq m\leq n$ and $\vec{k}\in\N^m$:
$$
\psi^{(2m+1)}_\lambda(X_i)\ket{\vec{k}}=
\left\{\!\!\begin{array}{ll}
q^{2(k_1+\ldots+k_i)}\ket{\vec{k}} &\text{if}\;i\leq m,\\[3pt]
0 &\text{if}\;i>m.
\end{array}\right.
$$
Then, for each $\alpha\in\N$ and index $1\leq i\leq n$, the sequence
$$
q^{-2\alpha}X_i\prod_{\substack{\beta=0,\ldots,N \\[1pt] \beta\neq \alpha}}\frac{X_i-q^{2\beta}}{q^{2\alpha}-q^{2\beta}}
\xrightarrow[\;N\to\infty\;]{} P(X_i,q^{2\alpha})
$$
is norm convergent to the spectral projection $P(X_i,q^{2\alpha})$ projecting onto the eigenspace of $X_i$ associated to the
eigenvalue $q^{2\alpha}$.
For $\vec{\alpha}=(\alpha_1,\ldots,\alpha_m)\in\N^m$ with $1\leq m\leq n$, let
$$
P_m(\vec{\alpha})=\prod_{i=1}^mP(X_i,q^{2(\alpha_1+\ldots+\alpha_i)}) \;.
$$
We now compute the pairing between the class of the projection $P_m(\vec{\alpha})$ in $K_0(C(\WP_q(\mv{\ell}))$
and the class of the Fredholm module $\mc{F}_{k,\vv{r}}$ of \S\ref{sec:Khom}.

\begin{thm}\label{thm:9.1}
Let $1\leq h\leq n$, let $\vv{r}=(r_0,\ldots,r_{h-1})$ satisfy \eqref{eq:ri},
let $1\leq m\leq n$, and $\vec{\alpha}=(\alpha_1,\ldots,\alpha_m)\in\N^m$.
If $h\geq m$ and $\alpha_{i+1}-r_i\in p_i\N$ for all $0\leq i<m$, the pairing is
$$
\inner{[\mc{F}_{h,\vv{r}}],[P_m(\vec{\alpha})]}=
(-1)^m\binom{N(\vv{r},\vec{\alpha})}{h-m}
$$
where the integer $N(\vv{r},\vec{\alpha})$ is given by
$$
N(\vv{r},\vec{\alpha})=\sum_{0\leq i<m}\frac{\alpha_{i+1}-r_i}{p_i} \;,
$$
and by convention a binomial $\binom{r}{s}$ is zero if $s>r$.
In all other cases, the pairing is zero.
\end{thm}

\begin{proof}
From Definition~\ref{def:15}, the pairing is 
$$
\inner{[\mc{F}_{h,\vv{r}}],[P_m(\vec{\alpha})]}
=\tr_{\HH_n}\bigg\{\sum_{\substack{0\leq k\leq h\\[1pt] k\;\text{even}}}\pi^{(h)}_k-
\sum_{\substack{0\leq k\leq h\\[1pt] k\;\text{odd}}}\pi^{(h)}_k\bigg\}\big(P_m(\vec{\alpha})\big) \;.
$$
We can start the sum from $m$, rather than $0$, because if $k<m$, the operator $y_m$ is in the kernel of the representation
and then $\pi^{(h)}_k(P(y_m,q^\beta))$ vanishes for any $\beta\in\N$. So, the pairing is $0$ if
$h<m$, while for $h\geq m$:
$$
\inner{[\mc{F}_{h,\vv{r}}],[P_m(\vec{\alpha})]}
=\tr_{\HH_n}\sum_{m\leq k\leq h}(-1)^k\pi^{(h)}_k\big(P_m(\vec{\alpha})\big) \;.
$$
Recall that $\HH_n$ is the orthogonal direct sum of $\mc{V}_k^h\cap\mc{V}_{k-1}^h$ over $k=1,\ldots,n$, plus the
joint kernel of all representations, and note that in fact $\mc{V}_0^h\subset\mc{V}_1^h$ and $\mc{V}_h^h\subset\mc{V}_{h-1}^h$.

A basis vector $\ket{\vec{m}}\in\mc{V}_k^h$ is in the range of $\pi^{(h)}_k\big(P_m(\vec{\alpha})\big)$
if and only if
$$
q^{2\sum_{j=1}^ir_{j-1}+p_{j-1}(m_j-m_{j-1})}=q^{2\sum_{j=1}^i\alpha_j} \quad\forall\;i=1,\ldots,m \;,
$$
which is equivalent to
\begin{equation}\label{eq:rap}
r_i+p_i(m_{i+1}-m_i)=\alpha_{i+1} \quad\forall\;0\leq i<m \;.
\end{equation}
So the range is zero (in all representations $\pi^{(h)}_k$) unless $\alpha_{i+1}\geq r_i$ and
$p_i$ divides $\alpha_{i+1}-r_i$, for all $0\leq i<m$. Assume this is satisfied.

For all $1\leq k\leq h-1$, from \eqref{eq:mconstr} and \eqref{eq:capconstr},
$$
\mc{V}_k^h=(\mc{V}_k^h\cap\mc{V}_{k-1}^h)\oplus(\mc{V}_{k+1}^h\cap\mc{V}_k^h) \;.
$$
For $m<k\leq h$, $\ket{\vec{m}}\in\mc{V}_k^h\cap\mc{V}_{k-1}^h$ is in the range of $\pi^{(h)}_k\big(P_m(\vec{\alpha})\big)$ if and only if
it is in the range of $\pi^{(h)}_{k-1}\big(P_m(\vec{\alpha})\big)$, and the total contribution to the pairing is zero. Similarly,
vectors in $\mc{V}_{k+1}^h\cap\mc{V}_k^h$, for $m\leq k<h$, give no contribution.
Non-zero contributions then can only come from the restriction of
$\pi^{(h)}_m\big(P_m(\vec{\alpha})\big)$ to $\mc{V}_m^h\cap\mc{V}_{m-1}^h$:
$$
\inner{[\mc{F}_{h,\vv{r}}],[P_m(\vec{\alpha})]}
=(-1)^m\tr_{\mc{V}_m^h\cap\mc{V}_{m-1}^h}\pi^{(h)}_m\big(P_m(\vec{\alpha})\big) \;.
$$
It remains to compute the dimension of the range of the projection.

A basis vector $\ket{\vec{\beta}}\in\mc{V}_m^h\cap\mc{V}_{m-1}^h$ in the range of $\pi^{(h)}_m\big(P_m(\vec{\alpha})\big)$
has labels $\beta_1,\beta_2,\ldots,\beta_m$ fixed, and from \eqref{eq:capconstr} the dimension is the number of
$(\beta_{m+1},\ldots,\beta_h)$ satisfying
$$
1\leq \beta_h +1< \beta_{h-1}+1<\ldots< \beta_{m+1} +1\leq \beta_m  \;.
$$
This is the number of $h-m$ partitions of $ \beta_m $, that is $\binom{\beta_m}{h-m}$
(it is zero if $h-m>\beta_m$).

From \eqref{eq:rap} and \eqref{eq:inverse}, we get
$$
\beta_m=\sum_{j=1}^m\frac{\alpha_j-r_{j-1}}{p_{j-1}} 
$$
and this concludes the proof.
\end{proof}

\begin{thm}
The K-homology classes $[\mc{F}_{h,\vv{r}}]$ are linearly independent over $\Z$.
\end{thm}
\begin{proof}
Suppose we have elements $\{e_i\}$ of a $\Z$-module $M$ and $\Z$-linear maps $f_i:M\to \Z$
satisfying $f_i(e_j)=\delta_{ij}$. Then the maps $f_i$ are linearly independent: if $\sum_in_if_i=0$,
then \mbox{$\sum_in_if_i(e_j)=n_j=0$} proving that all coefficients are zero.
We need to find a family of K-theory classes of projections dual to
the K-homology classes of the Fredholm modules $\mc{F}_{h,\vv{r}}$.

The trivial projection is dual to $\mc{F}_{0,\emptyset}$: $\inner{[\mc{F}_{h,\vv{r}}],[1]}$
is $1$ if $h=0$ and is $0$ otherwise.

Let now $1\leq h\leq n$, $1\leq m\leq n$,
$\vv{r}=(r_0,\ldots,r_{h-1})$ and
$\vv{s}=(s_0,\ldots,s_{m-1})$ chosen to satisfy
$$
0\leq r_i<p_i \;,\qquad\quad 0\leq s_j<p_j \;,
$$
as in \eqref{eq:ri}, for all $0\leq i<h$ and $0\leq j<m$.
For $\vec{\beta}=(\beta_1,\ldots,\beta_m)\in\N^m$ set $\beta_0:=0$ and
$$
\alpha_{i+1}(\vv{s},\vec{\beta})=s_i+p_i(\beta_{i+1}-\beta_i) \;,\qquad\forall\;0\leq i<m \;.
$$
From Theorem \ref{thm:9.1}, $\big<[\mc{F}_{h,\vv{r}}],[P_{m}(\vec{\alpha}(\vv{s},\vec{\beta}))]\big>$
is zero unless $h\geq m$ and $\alpha_{i+1}-r_i$ is divisible by $p_i$ --- i.e.~$s_i-r_i$ is divisible
by $p_i$ --- for all $0\leq i<m$. Since $|r_i-s_i|<p_i$, the latter condition implies $s_i=r_i$ 
for all $0\leq i<m$. Therefore, by Theorem \ref{thm:9.1}:
\begin{equation}\label{eq:par}
\big<[\mc{F}_{h,\vv{r}}],[P_{m}(\vec{\alpha}(\vv{s},\vec{\beta}))]\big>=
\begin{cases}
0 & \text{if}\;h<m,\\[2pt]
\delta_{r_0,s_0}\ldots\delta_{r_m,s_m}
(-1)^m\binom{\beta_m}{h-m} & \text{if}\;h\geq m,
\end{cases}
\end{equation}
where by convention $\binom{\beta_m}{h-m}=0$ if $h-m>\beta_m$. If we choose $\vec{\beta}=0$,
$$
\big<[\mc{F}_{h,\vv{r}}],[P_{m}(\vec{\alpha}(\vv{s},0))]\big>=(-1)^m\delta_{h,m}\delta_{r_0,s_0}\ldots\delta_{r_m,s_m} \;.
$$
We have our dual family of projections: the pairing is zero unless $h=m$ and $\vv{r}=\vv{s}$.
\end{proof}

The K-theory of weighted projective spaces is known.
As an abelian group it is independent of the weights \cite[Thm.~3.4]{Amr94}: 
$K^0(\WP(\mv{\ell})) \simeq K^0(\CP^n) \simeq\Z^{n+1}$.
On the other hand, the multiplicative structure making $K^0(\WP(\mv{\ell}))$ 
a ring does depend upon the weights \cite[\S5]{Amr94}. 

Notably, the K-theory of the quantum weighted projective spaces does not agree with the K-theory of their commutative counterparts.
For $n=1$, it was shown in \cite{BF12} that $K_0(C(\WP_q(\ell_0,\ell_1)))=\Z^{\ell_1 +1}$, 
while as said, in the commutative case $K^0(\WP(\ell_0,\ell_1))) = \Z^2$.

From the pairings computed in \eqref{eq:par}, one deduces that in the quantum case the $K_0$ and $K^0$ groups are bigger
for $n\geq 1$. Indeed, both $K_0(C(\WP_q(\mv{\ell})))$ and $K^0(C(\WP_q(\mv{\ell})))$
contain a subgroup isomorphic to $\Z^N$, where, for $\mv{\ell}=\mv{p}^{\,\sharp}$ and $\mv{p}$ pairwise coprime, 
$N$ is the number in \eqref{eq:number}:
$$
N=1+\sum_{k=1}^np_0p_1\ldots p_{k-1} \,.
$$


\vfill\eject

\end{document}